\documentclass[10pt]{article}

\usepackage{latexsym}
\usepackage{amsmath,amssymb, amsthm}
\usepackage{mathrsfs}

\usepackage{amsfonts}

\newtheorem{theorem}[equation]{Theorem}

\newtheorem{proposition}[equation]{Proposition}
\newtheorem{lemma}[equation]{Lemma}

\newtheorem{corollary}[equation]{Corollary}
\theoremstyle{definition}
\newtheorem{definition}[equation]{Definition}

\newtheorem{example}[equation]{Example}
\newtheorem{remark}[equation]{Remark}
\newtheorem{remarks}[equation]{Remarks}
\newtheorem{Notation}[equation]{Notation}

\newcommand{\FFF}{\mathbb F}
\newcommand{\ZZZ}{\mathbb Z}
\newcommand{\KK}{\mathbb K}

\newcommand{\DHO}{{\rm DHO }}

\newcommand{\LLL}{{\mathscr L}}

    \font\Aaa=msam10
\def\qed{\hbox{~~\Aaa\char'003}}

\def\col{\colon\!}

\numberwithin{equation}{section}

\begin{document}

\author{Ulrich Dempwolff and William M. Kantor}

\title{Orthogonal  Dual Hyperovals, Symplectic Spreads and Orthogonal Spreads}

\date{}
\maketitle

\begin{abstract} 
Orthogonal spreads in
orthogonal spaces of type $V^+(2n+2,2)$ produce large numbers of 
rank $n$  dual hyperovals in orthogonal spaces
of type $V^+(2n,2)$. 
The construction resembles the method for obtaining  symplectic spreads in $V(2n,q)$  from
orthogonal spreads in
 $V^+(2n+2,q)$ when $q$ is  even.
\end{abstract}

\section{Introduction}

A set ${\bf D}$ of $n$-dimensional subspaces 
spanning a finite
 $\FFF_q$-vector space  $V$    is called a 
\emph{dual hyperoval} (DHO)  {\em  of rank $n>2$},  if $| {\bf D} | =(q^n-1)/(q-1)+1$, $\dim X_1\cap X_2=1$ and 
$X_1\cap X_2 \cap X_3=0$ for every three different $X_1,X_2,X_3\in {\bf D}$. 
Usually DHOs are viewed 
projectively and called 
 ``dimensional dual hyperovals", but the vector space point of view seems better for our purposes.
See the survey article~\cite{Yo}
for many of the known DHOs, all of which  occur in vector spaces of  characteristic 2
and mostly are   over $\FFF _2$, 
in which case 
 $|{\bf D}|=2^n$.

Our purpose is to show that the number of rank $n$ 
\emph{orthogonal}
DHOs 
is not bounded above
by any polynomial in $2^n$;
these DHOs  occur in orthogonal  spaces $V^+ (2n,2)$
 and  all
members  are totally singular.
Our DHOs  will  
have a further property: they
\emph{split over a totally singular space $Y$}, meaning that
 $V=X\oplus Y$ for each DHO member $X $.
For more concerning the  number  of
inequivalent DHOs of rank $n$, see  Section~\ref{Concluding  remarks}(b).

Our  source for such orthogonal   DHOs 
in $V^+ (2n,2)$ is \emph{orthogonal spreads}
in $V^+(2n+2,2)$:  sets ${\bf O}$   of 
totally singular $n+1$-spaces such that each nonzero singular vector is in exactly one of them. Such orthogonal spreads exist if and only if $n$ is odd.  We use these for the following elementary result that is the basis for this paper: 

\begin{theorem} 
\label{ProjectionDHO}
Let ${\bf O}$ be an orthogonal spread in  $V^+(2n+2,2)$.
Let $P$ be a  point of $Y\in {\bf O},$ so that 
$V :=P^\perp /P\simeq V^+ (2n,2)$.
Then  
$$
{\bf O}/P:=\big \{ 
 \langle X\cap P ^\perp ,P  \rangle/P  \,| \,  X\in {\bf O}-\{Y\}\big\}
$$
is an orthogonal {\rm DHO} in $V$ that splits over $Y/P$.

\end{theorem} 

Although we will show that many orthogonal DHOs
can be obtained from orthogonal spreads with
the help of Theorem~\ref{ProjectionDHO}, there
are orthogonal DHOs   that   cannot be obtained by this method
(see   Section~\ref{Concluding  remarks}(a)).

Except in Section~\ref{qDHOs}, 
$q$ will always denote a power of 2
and almost always $n$ will be odd.
Our construction involves
 the close connection between
orthogonal spreads in $V^+(2n+2,q)$ and symplectic spreads in
$V(2n,q)$. Recall that a spread of $n$-spaces in $V=V(2n,q)$ is 
a set of $q^n+1$ subspaces such that each nonzero vector is in exactly one of them; this determines an affine plane \cite[p.~133]{Demb}.
A spread is called \emph{symplectic}
if there is a nondegenerate alternating bilinear form on  $V$ 
such that all members of the spread are totally  isotropic. 
Any symplectic spread 
in $V(2n,q)$ can be lifted to an essentially unique orthogonal spread
in $V^+(2n+2,q)$; conversely,   any orthogonal spread
in $V^+(2n+2,q)$ can be projected (in many ways, corresponding to arbitrary nonsingular points) in order to obtain symplectic spreads   \cite[Sec.~3]{Ka1}, \cite[Thm.~2.13]{KW}
(cf. Definition~\ref{Orthogonal and symplectic spreads} below).
Theorem~\ref {ProjectionDHO}  produces   many DHOs.  There is at present no determination of the number of inequivalent orthogonal spreads, and 
the same is true for DHOs.  


There is a simplified (and restricted) version 
of this process  that  does not take a detour using   orthogonal spreads of higher-dimensional spaces.
Given a symplectic spread ${\bf S}$ and distinct $X,Y\in {\bf S}$ it is standard to introduce 
``coordinates'':  a {\em spread-set}   $\Sigma$  for  ${\bf S}$
(this is  a set
of self-adjoint linear operators). These coordinates can be distorted in a unique
way to a set  $\Delta _\Sigma$ of coordinates  of an orthogonal DHO
(this is  a set
of skew-symmetric operators; see Theorem~\ref{Algebraic}), 
  which we   call a
   {\em shadow}
 of ${\bf S}$.
In some situations there are natural choices for $X$ or $Y$. For example,
if ${\bf S}$ defines a semifield plane then we let $Y $  be the shears
axis;  the  semifield spreads
in  \cite{KW}  produce the following

\begin{theorem}
\label{ShadowSemifieldKW} 
For odd composite $n$  there are more than
$2^{n(\rho (n)-2)}/n^2$ pairwise inequivalent orthogonal
{\rm DHO}s  in $V^+(2n,2)$  that are shadows of symplectic semifield spreads.
\end{theorem}

Here $\rho (n)$ denotes the number of (not necessarily distinct) prime factors of the integer $n$.
\emph{The number in the theorem is not bounded above by any polynomial in
$2^n$.} The proof uses a somewhat general isomorphism result  (Theorem~\ref{IsoShadow})     for DHOs arising from Theorem~\ref{ProjectionDHO}.

We also consider the symplectic
spreads ${\bf S}$  of the nearly flag-transitive  planes in 
\cite{KW1}. Here the automorphism group of ${\bf S}$  contains
a normal cyclic group fixing precisely two members of ${\bf S}$
and   acting regularly on the remaining ones, which leads to the following

\begin{theorem}
\label{ShadowNearlyFlagKW}
For   odd    composite  $n>27$  
 there are more than
$2^{3^{\rho (n)-1}}$ pairwise inequivalent orthogonal
{\rm DHO}s in $V^+(2n,2)$ admitting a cyclic group of order $2^n-1$ that  fixes 
one member of the \DHO and acts regularly on the remaining ones.
\end{theorem}

This time the number of DHOs is less than $2^n$.
We emphasize that there are many DHOs constructed using Theorem~\ref{ProjectionDHO} not considered in the preceding two theorems
(see   Example~\ref{transitive DHOs}). 

 In Section~\ref{qDHOs} we discuss a generalization of all of these results 
 to the more general context of $q$DHOs.

The authors of this paper view spreads and DHOs in somewhat different
 manners:  the first author prefers to think in terms of sets of operators \cite{ De1, DM}, while the second prefers sets of subspaces and (often) quasifields 
\cite{Ka1, Ka3, KW0, KW, KW1}.  We have mostly used the first approach (Theorem~\ref{ProjectionDHO} being the main exception), and have tried to provide translations between the two points of view 
(Remarks~\ref{CCKS skew-symmetric construction},
\ref{canonical=T condition} and
\ref{Constructing DHO-sets using orthogonal spreads},
Example~\ref{CCKS} and  Theorem~\ref{Clone}).



\section{Orthogonal DHOs and Theorem~\ref{ProjectionDHO}}

\label{Orthogonal DHOs}

All fields will have characteristic 2 except in Section~\ref{qDHOs}. 
Theorem~\ref{ProjectionDHO} is  sufficiently elementary that almost no background is needed:
\medskip

{\noindent \em Proof of} Theorem~\ref{ProjectionDHO}. 
It is standard that $V=P^\perp /P$ is an orthogonal space of type $V^+(2n,2)$ 
and that each totally singular subspace $X$ of $P^\perp$
has  a totally singular image  $\overline X$ in $V$.
  In particular, all members of ${\bf O}/P$
 are totally singular of dimension $n$. Since  $|{\bf O}/P|=2^n$ it suffices to show that any two
members of  ${\bf O}/P$ intersect in a point and
 any three   intersect trivially.
 
 Let  $X_1,X_2,X_3\in {\bf O}-\{ Y\}$
 be distinct.  Then $\overline X_i=
 \langle X_i\cap P ^\perp ,P  \rangle /P$
 and $\overline Y= Y /P$.
 Let $P=\langle w\rangle$.

 Since $(X_1\cap P^\perp )\cap (X_2\cap P^\perp )=0$ 
 we have $\dim \overline{X}_1\cap \overline{X}_2 \leq 1$. 
 On the other hand,   $w=x_1+x_2$ for some
$0\neq x_i\in X_i$. 
All vectors in the 2-space $\{0,w,x_1,x_2\}$ are singular, so this is a totally singular 2-space. 
 Hence $x_i\in X_i\cap P^\perp$ and
 $
 \overline{X}_1\cap \overline{X}_2=\langle x_1,w\rangle /P
 =\langle x_2,  w \rangle /P 
 $
  has dimension $1$, as required.
  
  Similarly,   
  \vspace{1pt}
  $
 \overline{X}_1\cap \overline{X}_3 
  =\langle x_3,  w \rangle /P 
 $ with
 $w=x_1'+x_3$, 
 $0\ne x_1'\in X_1$,
 $0\neq x_3\in X_3$.  
  
  If  $\overline{X}_1\cap \overline{X}_2\cap \overline{X}_3\ne0$ then
    $\overline{X}_1\cap \overline{X}_2=\overline{X}_1
  \cap \overline{X}_3$,  so that 
  $\{0,w,x_1,x_2\} =
\langle x_1,w\rangle 
=\langle x_1 ' ,w\rangle
  = \{0,w,x'_1,x_3 \}$  (our field is $\FFF_2$!). 
  This is impossible, since 
  $0\ne x_2\in X_2$ whereas 
  $X_2$ intersects $Y,X_1$ and $X_3$ only in 0.
Thus,   ${\bf O}/P$  is a DHO.%
\vspace{1pt}

  Finally,  if $x+P = y+P$ lies in 
$\overline{X_1}\cap \overline{Y}$ ($x\in X_1$, $y\in Y$),   
then   $x\in X_1\cap  (y+  P) \subseteq X_1 \cap Y  =0$, so that  
$ {\bf O}/P$  splits over
  $\overline{Y}$.\qed

 \begin{definition}
 \label{DHO projections}
The DHO  ${\bf O}/P$  in
{\rm Theorem~\ref{ProjectionDHO}}  is    the \emph{projection of ${\bf O}$ with respect to  $P$}.   Note that $Y\in {\bf O}$ is determined
by $P$. 
\end{definition}
 
The notions of equivalence and 
automorphisms of   symplectic or
 orthogonal spreads, and of DHOs, are crucial for our results:

  \begin{definition}
\label{Equivalence} 
$T\in \Gamma {\rm L}(V)$  is an 
\emph{equivalence}
  ${\bf E}\to {\bf E}'$ between  sets  
  ${\bf E}$  and $  {\bf E}'$
  of subspaces
of a vector space $V$
 if $T$ sends ${\bf E}$
onto ${\bf E}'$. The
\emph{automorphism group ${\rm Aut}({\bf E})$ of ${\bf E}$}
is
the group of equivalences from ${\bf E}$
to itself.
 \end{definition}
 
 Clearly, in Theorem~\ref{ProjectionDHO}  points $P$
  in the same ${\rm Aut}({\bf O})$-orbit produce isomorphic DHOs  ${\bf O}/P$
and   the stabilizer ${\rm Aut}({\bf O})_P$ of $P$ induces an automorphism group of
 ${\bf O}/P$.

Our goal   is the construction of large numbers of inequivalent DHOs.
For this purpose we need to  compare the construction in 
Theorem~\ref{ProjectionDHO}  to ones in    \cite[Sec.~3]{Ka1} and \cite[Thm.~2.13]{KW}
  (cf. Section~\ref{Projections and lifts with coordinates}).  First we recall another standard property of orthogonal spaces $V^+(2m,q)$ 
\cite[Thm.~11.61]{Tay}:  the set of totally singular $m$-spaces 
is partitioned into two equivalence classes where
   totally singular $m$-spaces $X,Y$ 
   are equivalent  if and only if $\dim X\cap Y \equiv m$ (mod 2).

 \begin{definition} [\emph{Lifts and projections of symplectic  and 
 orthogonal  spreads}]
\label{Orthogonal and symplectic spreads}
Assume that $n$ is odd.
Let $N $ be a nonsingular point of 
$\overline V=V^+(2n+2,q)$,
so that  $V:=N ^\perp/N  \simeq  V(2n,q)$ is a symplectic space.  
 If ${\bf S}$ 
 is a symplectic spread in $V$
	and ${\cal M}$ is one of the two classes of totally singular $(n+1)$-spaces in $\overline V$,  then (since $n+1$ is even) 
	 $$
	 \{ X\in {\cal M}\,| \, \langle X\cap N ^\perp ,N  \rangle /N  \in {\bf S} \}
	 $$
	 is an orthogonal spread in $\overline V$,
	 the {\em  lift} of {\bf S}.  $($Changing  ${\cal M}$  produces an equivalent orthogonal spread.$)$

	 This reverses$:$  if ${\bf O}$ is an orthogonal spread in 
	$V^+(2n+2,q)$,
then
  $$
{\bf O}/N  :=\big \{  \langle X\cap N ^\perp ,N  \rangle/N  \,| \, X\in {\bf O} \big\}
$$
is a symplectic spread in $V,$
	 the {\em projection} of  ${\bf S}$ with respect to $N$. 
	 This  strongly resembles Definition~\ref{DHO projections}.
As before,   points $N$
  in the same ${\rm Aut}({\bf O})$-orbit produce isomorphic spreads  ${\bf O}/N$.

\end{definition}  
 
\begin{proposition}
\label{Basic}
Let ${\bf D}$ be an orthogonal \DHO of $V=V^+(2n,2)$.
 Then

\begin{enumerate}
	\item[\rm (a)] $n$ is odd$,$ and 
	\item[\rm (b)] If ${\bf D}$ splits over a totally singular
	subspace $Y,$  then  $\bigcup _{X\in {\bf D}}X \cup Y$
	is the set of singular vectors in $V$. In particular$,$
	$Y$ is the only totally singular
	subspace over which ${\bf D}$ splits.
	 
\end{enumerate}

\end{proposition}

\begin{proof} (a)
If  $X_1,X_2,X_3$ are distinct members of  ${\bf D}$, then any two have intersection of dimension 1.  
If $n$ is even,  then any two lie in different  classes
of   totally singular $n$-spaces, whereas   there are only two such classes.
\smallskip

(b)
 The set  $S_V$   of singular vectors of $V$ has size $2^{2n-1}+2^n$. 
By Inclusion-Exclusion,  $|\bigcup _{X\in {\bf D}}X|=2^{2n-1}+1$. Thus, if ${\bf D}$ splits over
the totally singular 
subspace $Y$
then $Y-0=S_V-\bigcup _{X\in {\bf D}}X$.
\end{proof}

\begin{remark}
 \emph{We will exclusively deal with orthogonal
{\rm DHO}s that split over  totally singular subspaces.}
However, there are  orthogonal   DHOs      
(see  {\rm  \cite[Prop. 5.4]{De1}})
that   split over subspaces that are not totally singular
but 
do not split over any totally singular subspace. \end{remark}


Any linear operator
preserving   an orthogonal DHO lies in the orthogonal group:

\begin{proposition} 
\label{Focus} Let ${\bf D}$ and 
${\bf D}'$ be  orthogonal 
{\rm DHO}s of  $V=V^+(2n,2)$ 
 that split  over the totally
  singular   subspace $Y$.
If $\Phi \in {\rm GL}(V)$ sends 
${\bf D}$ to  ${\bf D}'$, 
 then $\Phi$ lies in the stabilizer of  $Y$
  in the orthogonal group ${\rm O}(V)$.
\end{proposition}

\begin{proof} 
By   Proposition~\ref{Basic}(b), 
$S_V=\bigcup _{X\in {\bf D}} X \cup Y$
is the set of  {all} singular vectors in $V$.
Every $3$-dimensional subspace that has exactly six points in  
$S_V $ not in $Y$  is totally singular and hence 
has a seventh point in $Y$. Since every point of $Y$ arises this way,
$\Phi$ leaves $S_V$ and $Y$ invariant.
\end{proof}
 
\begin{corollary}
\label{Contragredient}
Let  
  ${\bf D}$ be an orthogonal 
  \DHO of  
  $V=V^+(2n,2)$  that splits over the totally
  singular subspace $Y$. 
Then $Y$ is invariant under 
  $G={\rm Aut}({\bf D}),$
  and the representation of $G$ induced on $V/Y$   is contragredient
  to the representation of $G$ induced on $Y$.

\end{corollary}

\begin{proof}
The preceding  proposition  implies the  first assertion.
Since  $Y=Y^\perp$, the bilinear form associated  with the quadratic form
induces a $G$-invariant duality from $Y$ onto $V/Y$,
which implies    the second assertion.
\end{proof}



\section{Coordinates and  symplectic spread-sets}
\label{Coordinates and  symplectic spread-sets}

In this section we use coordinates of orthogonal and symplectic spreads
in order to describe
 operations   
that do not require   projections from higher-dimensional  orthogonal
spreads. 
Throughout the remainder of this paper we will always have 
$$   \mbox{$U=V(n,q)$ \  and   \  $V\cong U\oplus U$},
$$
 where $q$ is  even except in Section~\ref{qDHOs}.
If $U$ is equipped with a nondegenerate symmetric bilinear form $ \textsf{b}(\cdot ,\cdot)$
we denote by $T^\star$ the  operator   adjoint to $T\in {\rm End}(U)$.

\subsection{\noindent\bf Coordinates for symplectic spreads, orthogonal spreads
and orthogonal DHOs} 
\label{coordinates}
Assume that $V $ is a   symplectic  space,
 and denote by ${\bf E}$ either
a symplectic spread, an orthogonal spread, or an orthogonal
DHO in $V$ that splits over a totally singular subspace.  The  symplectic form $(\cdot ,\cdot) $ on $V$  vanishes on all members of ${\bf E}$.  
For an orthogonal spread or  DHO,
  all members of ${\bf E}$ are totally singular
  with respect to a quadratic form $Q$ polarizing to
$(\cdot ,\cdot) $.  For  a DHO we always assume 
that $q=2$.

In order to coordinatize  ${\bf E}$ we choose any distinct
$X,Y\in {\bf E}$ if ${\bf E}$ is a symplectic or orthogonal spread.
If ${\bf E}$  is a   DHO that splits over a totally singular subspace  $Y$ then choose $X\in {\bf E}$.
 We identify $V$ with $U\oplus U$.  We may assume that
   $$
   \mbox{$X=U\oplus  0$ and $Y=0\oplus  U$,}
   $$
     the symplectic  form on  $U\oplus U$  is
\begin{equation}
\label{bilinear form}
 \big  ((x,y),(x',y')\big)= \textsf{b}(x,y')+ \textsf{b}(y,x') ,
\end{equation}
  and the quadratic form
is  
\begin{equation}
\label{quadratic form}
 Q\big((x,y)\big)=\textsf{b}(x,y).
 \end{equation} 

 For  $Z\in {\bf E}-\{ Y\}$ there is a unique
   $L\in {\rm End}(U)$ such
 that $Z=V(L)$, where
\begin{equation}
\label{V}
V(L) :=\{ (x,xL) \,| \, x\in U\}.
 \end{equation} 
Each
 $L$ is self-adjoint with respect to $\textsf{b}$ if ${\bf E}$ is a symplectic spread
(as $Z$ is totally isotropic), and $L$ is even skew-symmetric
(i.\,e., $\textsf{b}(x,xL)=0$ for all $x$) 
if ${\bf E}$ is an orthogonal spread or  a DHO (as $Z$ is totally singular).
The subspace $Z=X$ corresponds to $L=0$.  If $Z\neq X$ then
$L$ is invertible if ${\bf E}$ is a symplectic or orthogonal spread, while
$L$ has rank $n-1$ in the DHO case. 
Hence, there is a set $\Xi
\subseteq {\rm End}(U)$ containing 0  
such that 
$$
{\bf E}=\{ V(L) \,| \, L\in \Xi \}\cup \{ Y\}
$$
if ${\bf E}$ is a symplectic or orthogonal spread
and 
$${\bf E}=\{ V(L) \,| \,  L\in \Xi \}
$$ 
if ${\bf E}$ is  an  orthogonal DHO
that splits over the totally singular subspace $Y=0\oplus  U$.

\medskip

\begin{definition} Let $V=U\oplus U, {\bf E}, X$ and $Y$ be as above.
\begin{itemize}
\item
If ${\bf E}$ is a symplectic  spread, then
 $\Xi$ is a $ ($symplectic$)$ 
 \emph{spread-set of ${\bf E}$ with respect to the ordered pair
  $(X,Y)$}.
  \item
If ${\bf E}$ is an orthogonal spread, then
  $\Xi$  is a  
 \emph{Kerdock set of ${\bf E}$ with respect to the ordered pair
  $(X,Y)$} (cf. \cite{Ka1}).
  \item
If ${\bf E}$ is an orthogonal DHO then $\Xi$ is a  
  DHO-\emph{set of ${\bf E}$ with respect to 
  $X$}.  (Note that there is no choice for $Y$,   the  space over which
${\bf E}$  splits.)
  \end{itemize}
 \end{definition}
  

 Conversely,  it is routine to check the following: 
 \begin{lemma}
 \label{Coord}
 Assume that   $\Xi \subseteq {\rm End}(U)$ is a set
 of self-adjoint operators  containing $0$.
Define symplectic  and  quadratic forms on  $V=U\oplus  U$   using \eqref{bilinear form} and  \eqref{quadratic form}.
\begin{enumerate}
\item[\rm (a)] If $|\Xi |=q^n$ and $\det (L+L')\neq 0$
	for  all distinct $L,L'\in \Xi,$  then
	${\bf E}=\{ V(L) \,| \, L\in \Xi \}\cup \{ 0\oplus  U\}$
	is a symplectic spread of $V$.
\item[\rm(b)] If $|\Xi |=q^{n-1},$  $\det (L+L')\neq 0$
	for  all distinct $L,L'\in \Xi,$  and  all members of 
	 $\Xi$ are skew-symmetric$,$  then
	${\bf E}=\{ V(L) \,| \,  L\in \Xi \}\cup \{ 0\oplus  U\}$
	is an orthogonal spread  of $V$.
\item[\rm(c)] Assume that $|\Xi |=2^n $ with 
$n$   odd$,$   that 
		all members of  $\Xi$ are skew-symmetric$,$ and  that
	\begin{enumerate}
  \item[\rm(1)] ${\rm rk}(L+L')=n-1$   for  all distinct $L,L'\in \Xi,$ and

  \item[\rm(2)] If $L\in \Xi$ then $\big \{ \ker (L+L')\,| \, L'\in \Xi -\{L\}\big \}$
  is the set of $1$-spaces of $U$.%

\end{enumerate}
Then
	${\bf E}=\{ V(L) \,| \, L\in \Xi \}$
	is an orthogonal \DHO that splits over  $0\oplus  U$.
\end{enumerate}

 \end{lemma}


\begin{remark}
\label{Product}
Let  $\textsf{b}(x,y)=x\cdot y$ be the usual dot product
and identify ${\rm End}(U)$ with the space of all $n\times n$ matrices over
$\FFF _q$.
Then $L\in {\rm End}(U)$ is self-adjoint
if and only if  $L=L^t$, and $L$ is  skew-symmetric
 if and only if,  in addition,  its diagonal
is  $0$.

A variation is   used in Sections 4 and 5: identify $U$ 
with $F=\FFF_{q^n}$ and use the trace form
$$
\textsf{b}(x,y)= {\rm Tr}(xy),
$$
where ${\rm Tr} \col  F \to \FFF _q$ is the  trace map.

\end{remark}

Rank 1 operators will play a crucial role for  our results.
The following elementary 
description of
those  operators is also  in \cite[Prop. 5.1]{LMPT1}.

\begin{lemma} 
\label{RankOne} 
  If
$T \in {\rm End}(U)$  has  rank $1,$
then  $T=E_{a,b}$
for some $0\neq a,b \in U,$  where  
\begin{equation}
\label{Im Eab}
\mbox{
$xE_{a,b}:=\textsf{b}(x,a)b$  \ for all \ $x\in U$.
}
\end{equation}
If $E_{a,b}=E_{a',b'}$  for nonzero $a,a',b,b',$
then  $a'=ka$
and $b'=k^{-1}b$  for some  $k\in \FFF _q^\star$.
\end{lemma}

\begin{proof}
 Write $U_0=\ker T = \langle a\rangle ^\perp$ with $a\in U$. Let $v\in U-U_0$    such that $\textsf{b} (v,a)=1$.
 Then $b:=vT\neq 0$.
Thus, $0=uT=uE_{a,b}$  for $u\in U_0$ and 
$vT=b=vE_{a,b}$, so that $T=E_{a,b}$.

If $E_{a,b}=E_{a',b'}$  then   $\langle a \rangle
=\langle a' \rangle$ and $\langle b \rangle
=\langle b' \rangle$, and a   calculation completes the proof.
\end{proof}

\begin{remark}
 Since $\textsf{b} (x ,yE_{a,b})=\textsf{b}(xE_{b,a},y)$,
 the operator
$E_{b,a}$ is adjoint to $E_{a,b},$ 
so that  $E_{a,b}$ is self-adjoint if and only if  $\langle a \rangle
=\langle b \rangle$. In this case there is  a $($uniquely determined$)$ 
$c\in \langle a \rangle
=\langle b \rangle$ such that $E_{a,b}=E_{c,c}$.

In terms of matrices, the lemma is the elementary fact that rank $1$ matrices 
have the form $a^tb$ for nonzero row vectors $a,b$.  This
matrix  is symmetric if and only if $\langle a \rangle
=\langle b \rangle$.

\end{remark}


\begin{lemma}
\label{SelfAdjoint}
For each self-adjoint operator  $T$ there  
  is a unique
  self-adjoint operator  $R=E_{a,a}$ of
rank $\leq 1$   such that $T+R$
is skew-symmetric. Moreover$,$

\begin{enumerate}
	\item[\rm (a)] $a\in {\rm Im}\,T;$ 

	\item[\rm(b)]	
	$
	{\rm rk}\, (T+R)=
	\left\{
	\begin{array}{ll}
	 {\rm rk}\, T & {\rm if \ } {\rm rk}\,    T\equiv 0\pmod{2}\\
	 {\rm rk}\, T  \pm 1 &{\rm if \ } {\rm rk}\,T\equiv 1\pmod{2};
	 \end{array}
	 \right.
	 $
	 
		\item[\rm(c)] If $S$ is self-adjoint and $S+E_{b,b}$
		is skew-symmetric$, $ then $R'=E_{a+b,a+b}$ is the unique
		self-adjoint operator of rank $\leq 1$ such that
		$T+S+R'$ is skew-symmetric$;$ and 
		
		\item[\rm(d)] If $n$ is odd and $T$ is invertible$,$
		then $\ker ( T+ E_{a,a})=\langle aT^{-1} \rangle$ and
		$\textsf {b} (a, aT^{-1})\neq 0$.%

\end{enumerate}

\end{lemma}

\begin{proof}
As $T$ is self-adjoint, the map $\lambda _T \col  U\to  \FFF_q$
given by  $ x\mapsto \textsf{b}(x,xT)$
is semilinear: $\lambda _T(kx)=k^2\lambda _T (x)$ for
$x\in U,k\in \FFF_q$. 
If $\lambda _T=0$ then $T$ is skew-symmetric and we set $R=0=E_{0,0}$.
Assume that $\lambda _T\neq 0$ and
set $U_0=\ker \lambda _T$.  Pick $u\in U$ such that $\lambda _T(u)=1$ and  
 $a\in U$ such that
$U_0=\langle a \rangle ^\perp$ and $\textsf{b} (u,a)=1$. Then $S=T+E_{a,a}$ is self-adjoint.
Moreover $\lambda _S(x) =\lambda _T(x) + \textsf{b} (x,a)^2$ is 0 on both  $U_0$ and $u$, so that  $S$ is skew-symmetric. In particular,
\begin{equation}
\label{b(x,a)}
\mbox{$\lambda _T(x)= \textsf{b} (x,a)^2$
for all $x \in U$. 
}
\end{equation}
As $\textsf{b}$ is nondegenerate,  
every semilinear functional from  $U$ to $\FFF_q$ associated
with the Frobenius automorphism 
has the form $x\mapsto  \textsf{b} (x,a)^2$ 
 for a  {unique}
$a\in U$. This implies the uniqueness of $R=E_{a,a}$.
\smallskip

 (a) Let $T+E_{a,a}$ be skew-symmetric and assume
 that  $a\not\in {\rm Im}\,T=({\rm Im}\,T)^\perp{}^\perp$. 
 Then $ \textsf{b} (a, ({\rm Im}\,T^\perp ))\neq \{0\}$,
 so that   there exists
  $y\in
 ({\rm Im}\,T)^\perp$ with 
 $1=\textsf{b} (a,y)$.   Since
 $y$ and $yT$ are perpendicular, (\ref{b(x,a)})  implies that 
 $1=\textsf{b} (a,y)^2 =\textsf{b} (y,yT)=0$,
 a contradiction.
 \smallskip
 
 (b) Clearly ${\rm rk}\, (T+R)\equiv 0\pmod{2}$. 
 \smallskip
 
 (c) 
 $
 (T+S)+E_{a+b,a+b}=(T+E_{a,a})+ (S+E_{b,b})+(E_{a,b}+E_{b,a})
 $
expresses  the left hand side   as a sum of skew-symmetric operators.
\smallskip

  (d)   By (b),  $\dim\ker(T+E_{a,a})=1$.
Let $0\neq x\in \ker T+E_{a,a}$. 
By (\ref{Im Eab}), 
$0=xT+\textsf{b} (a,x)a$ and hence  $x=\textsf{b} (a,x)aT^{-1}$, so that   
$0\ne x\in \langle aT^{-1}\rangle $ and $\textsf{b} (a,x)\neq 0$.\end{proof}

\begin{remark}
\label{CCKS skew-symmetric construction}
 In terms of matrices the first paragraph of the lemma states that, if $A$ is a symmetric matrix, then $A+d(A)^td(A)$ is skew-symmetric$,$ where 
$d(A)$ is the diagonal of $A$ written as a row vector 
as in 
 \cite[Lemma~7.3]{CCKS}.
\end{remark}

\begin{lemma}
\label{Labeling}
For a symplectic spread-set  $\Sigma$ of $U=V(n,q) $  with   $n$ odd$,$   
\begin{enumerate}
\item[\rm(a)] 
There is a unique bijection $C\col U\to \Sigma$ such that
 $C(a)+E_{a,a}$ is skew-symmetric for all $a\in U,$  and
 \item[\rm(b)]
$C$ is additive iff $\Sigma$ is additively closed.
\end{enumerate}

\end{lemma}

\begin{proof}
(a)
If  $0\neq L\in \Sigma$ then  the self-adjoint, invertible operator  $L$  is not    
skew-symmetric as $n$ is odd.  By the preceding lemma there is a unique nonzero
vector $a=a_L\in U$ such that $L+E_{a,a}$ is
skew-symmetric of rank $n-1$. If $0\neq L,L'\in \Sigma$,
$L\neq L'$, then $a_L\neq a_{L'}$ as $L+L'$ is 
 invertible and hence not skew-symmetric,  so that $C$ is bijective.
 
 (b) 
 Since one direction is obvious, assume that $\Sigma$ is additively closed.
 If $a,b\in U$,  then $C(a)+C(b)=C(c)$ for some $c\in U$.
 By definition $C(c)+E_{c,c}$ is skew-symmetric, and so is 
  $C(a)+ C(b)+E_{a+b,a+b}=C(c)+E_{a+b,a+b}$ by  Lemma~\ref{SelfAdjoint}(c).
Then $c=a+b$ by Lemma~\ref{SelfAdjoint},
as required.
\end{proof}

\begin{definition}[{{\em Canonical labeling}}]
\label{canonical}
The unique
bijection $C \col U\to \Sigma$ in  {\rm Lemma~\ref{Labeling}} is  the
\emph{canonical labeling} of the symplectic spread-set  $\Sigma$  of operators of  $U$.  Notation: 
 $C=\LLL(\Sigma)$.

\remark
\label{canonical=T condition}
Each symplectic spread-set $\Sigma\subseteq {\rm End} (U)$ determines 
a prequasifield on $U$ defined by $x*a=xC(a)$ for any additive bijection $C\col U\to \Sigma$.   Then $C$ is the canonical labeling if and only if
 $$
\textsf{b}(x,x*a) = \textsf{b}(x,xC(a))= \textsf{b}(x,x E_{a,a})
= \textsf{b}(x,a)^2 
$$ 
by (\ref{Im Eab}).~This is   the condition on 
a prequasifield appearing in \cite[(2.15)]{KW}.

\end{definition}

\subsection{Projections and lifts with coordinates} 
\label{Projections and lifts with coordinates} 

We next coordinatize 
 projections and lifts 
(Definitions~\ref{DHO projections} and ~\ref{Orthogonal and symplectic spreads}).
We review    \cite{Ka1, Ka3,KW} using somewhat different notation.  We will assume 
for the remainder of Section~\ref{Coordinates and  symplectic spread-sets} that 
\emph{$n$ is odd}.  

\medskip
\noindent
(a) {\sc From Kerdock sets to symplectic spread-sets.} \ Let ${\bf O}$ be an orthogonal spread in $\overline{V}
=V^+(2n+2, q)$,  let $N$ be a nonsingular
point, and choose an ordered pair  $X,Y\in {\bf O}$.
The  identification 
\begin{itemize}
	\item  $\overline{V}= \overline{U}\oplus  \overline{U}$ where
	   $\overline{U}=V(n+1,q)$,
	
	\item $X= \overline{U}\oplus  0$,  \ $Y=0\oplus  \overline{U}$,
		
\end{itemize}
produces a  Kerdock set
$\KK $
such that
each member of  ${\bf O}-\{Y\}$ has the form
$V(L)=\{ (x,xL)\,| \, x \in \overline{U}\}$, $L\in \KK $.
Moreover,  this identification induces a symmetric, nondegenerate 
bilinear form
 $\overline{\textsf{b}}(\cdot,\cdot)$ on $\overline{U}$ such that
the quadratic form $Q$ is defined by   $Q\big ((x,y)\big)=
\overline{\textsf{b}}(x,y)$.
Given this Kerdock set, we make the special choice  
$$
N=\langle (w,w)\rangle \  \ {\rm with } \  \  \overline{\textsf{b}}(w,w)=1.
$$
Then $(x,xL)$ lies in $N^\perp$ if and only if  
$\overline{\textsf{b}}(w,x)=\overline{\textsf{b}}(w,xL)$.
Set $U=\langle w\rangle ^\perp$ and write
$x\in \overline{U}$ as $x= \alpha w+u$, $\alpha \in \FFF _q$, $u\in U$.
As $L$ is skew-symmetric, $wL\in U$ and
$$
\alpha =\overline{\textsf{b}}(w,x)=\overline{\textsf{b}}(w,xL)=
\overline{\textsf{b}}(wL,u).
$$
Also,
$$
uL=uL\pi_U+ \overline{\textsf{b}}(wL,u)w,
$$
where $\pi _U$ is the orthogonal projection 
$\overline U \to U$.
Since $U\oplus  U$ is a set of representatives for $N^\perp/N$
and as 
$(x,xL)=
(\overline{\textsf{b}}(wL,u)w+u,\overline{\textsf{b}}(wL,u)w +
\overline{\textsf{b}}(wL,u)wL+uL\pi _U)\equiv (u, \overline{\textsf{b}}(wL,u)wL+uL\pi _U)$~(mod $N$),  
$$
\mbox{$
\{ L\pi _U +E_{wL,wL} \,| \,  L\in \KK  \}$ \ \em
is a spread-set of the symplectic spread $\,{\bf O}/N$.
}
$$

\noindent
(b) {\sc From Kerdock sets to DHO-sets.}
We keep the notation from (a) using $q=2$.
We use
  $X\in {\bf O}-\{Y\}$
  and   the singular point $P=\langle (0,w)\rangle \subseteq Y.$ 
We use the above identifications
for $\overline{V}$, $X$, $Y$ and $Q$.
A typical element in $V(L)\cap P^\perp$
has the form $(u,uL)=(u, uL\pi _U+\textsf{b}(wL,u)w)
\equiv (u, uL\pi _U)$~(mod $P$), $u\in U$.
As   
 $U\oplus  U\simeq P^\perp/P$, we see that
  $$
\mbox{$\{ L \pi _U  \,| \, L\in \KK  \}
$ \ \em
is a {\em DHO}-set of the orthogonal \DHO  $\,{\bf O}/P$.}
$$
\noindent
(c) {\sc From symplectic spread-sets to Kerdock sets.}
Let ${\bf S}$ be a symplectic spread on $V=V(2n,q)$, 
and let $X,Y\in {\bf S}$.
This time we identify
\begin{itemize}
	\item $V=U\oplus  U$, $U=V(n,q)$,
	\item $X=U\oplus  0$, $Y=0\oplus  U$, and 
	\item The bilinear form is 
	$\big ((x,y),(x',y')\big)=\textsf{b}(x,y')+\textsf{b}(y,x')$ 
	for a nondegenerate symmetric bilinear form   $\textsf{b}$
	on $U$.
\end{itemize}
Let $\Sigma\subseteq {\rm End}(U)$ be  the  resulting spread-set and  $C=\LLL(\Sigma)$ (cf.
Definition~\ref{canonical}).  
Set $\overline{U}=\FFF_q\oplus  U$ and 
$\overline{V}=\overline{U}\oplus  \overline{U}$, and define a
 quadratic form $Q$ on $\overline{V}$ by
$$
Q(\alpha , x, \beta , y)=
\alpha \beta +\textsf{b}(x,y) .
$$
For $a\in U$ define the
skew-symmetric linear operator $D(a)$ on $\overline{U}$ by
$$
(\alpha ,x)D(a)=\big( \textsf{b}(x,a),\alpha a+x(C(a)+E_{a,a})\big ).
$$
Then
$\KK  =\{D(a) \,| \, a \in U\}$ is a Kerdock  set of the lift
${\bf O} $,  where
${\bf O}/N\simeq {\bf S}$ for the choice  $N=\langle (1,0,1,0)
\rangle$.

\begin{example} 
\label{CCKS}
\rm
We illustrate the above discussion   using matrices,
as in  \cite[Lemma 7.3]{CCKS}.
Let $\overline{U}=\FFF _q^{n+1}$ and $\overline{V}=
\overline{U}\oplus  \overline{U}$, equipped with the quadratic form
$Q(x,y)=x\cdot y$.
We will use the nonsingular point  $N=\langle (e_1,e_1 )\rangle$
and  the singular point    $P=\langle (0,e_1 )\rangle$ (where the $e_i$ are the standard basis vectors of $\overline{U}$).
Then  the bilinear form $\textsf{b}$ is
 the usual dot product on $U:= \langle e_2, \ldots, e_{n+1}\rangle $.

Let ${\bf O}$ be an orthogonal spread containing $X$ and $Y$ (defined  above).
Then a Kerdock set  can be written $\KK  =\{ D(u) \,| \, u\in U\}$
using $(n+1)\times (n+1)$ skew-symmetric matrices 
$$
D(u)=\begin{pmatrix} 
        0 & x(u) \\
        x(u)^t & A(u)\\
     \end{pmatrix},
     $$
where
 $A(u)$ is  an $n\times n$ skew-symmetric matrix  and
 $x(u)\in U$ is  a row matrix.   Then
\begin{equation}
\label{Psi}
\Delta := \{ A(u)  \,| \, u \in U\}
\end{equation}
is a DHO-set of ${\bf O}/P$,  while
$$
\Sigma :=  \{ A(u)  + x(u)^tx(u)\,| \, u \in U\},
$$
is a   spread-set
of the symplectic spread ${\bf O}/N$,
where   $x(u)^tx(u)$ represents the previous
  rank 1 operator $E_{wL,wL}$ in (a).

\end{example}


\subsection{Shadows, twists and dilations} 
 \label{Shadows} 
\begin{theorem}
\label{Algebraic}

Let  $\Sigma $ be a spread-set of
self-adjoint operators of  $U=V(n,2)$  and $C=\LLL(\Sigma)$.
Then $\Delta = \Delta _\Sigma =\{ B(a)=C(a)+E_{a,a}\,| \,  a\in U \}$
is a \DHO\!\!-set of skew-symmetric operators.
\end{theorem}
 
 \begin{proof}
 We sketch two different arguments.

 {\sc Geometric approach.}  Start with a symplectic spread-set $\Sigma$ and
 $C=\LLL (\Sigma)$, and produce a Kerdock set $\KK $
 using Section~\ref{Projections and lifts with coordinates}(c). Then apply Section~\ref{Projections and lifts with coordinates}(b) to $\KK $  using
 the singular point  $P=\langle (0,0,1,0)\rangle$.
 
 {\sc Algebraic approach.}  We will verify the conditions in  Lemma~\ref{Coord}(c).  Consider distinct  $a,b,c\in U$.
Then skew-symmetric operator   
$B(a)+B(b)=C(a)+C(b) +E_{a,a}+E_{b,b}$ has even rank at least $n-2$,
and  hence has rank    $n-1$.
  
Let $x\ne0$ with  $x(B(a)+B(b))=x(B(a)+B(c))=0$.  Then $0\ne 
x(C(a)+C(b))=\textsf{b} (a,x)a+\textsf{b} (b,x)b$, so that 
$\textsf{b} (a,x)$ or  $\textsf{b} (b,x)\ne0$.
We cannot have $\textsf{b} (a,x) =\textsf{b} (b,x)=1$,
as otherwise $\textsf{b} (a+b,x) =0$
would contradict Lemma~\ref{SelfAdjoint}(d)
(since  $C(a)+C(b)+E_{a+b,a+b}$ is skew-symmetric by 
Lemma~\ref{SelfAdjoint}(c)).

Then  $\textsf{b} (a,x)\neq \textsf{b} (b,x)$. By symmetry,
  it follows that
$\textsf{b} (a,x),$ 
$\textsf{b} (b,x)$  and  
$\textsf{b} (c,x)$
are distinct members of $\FFF_2$, a contradiction. 
\end{proof}

\remark[{{\em Constructing {\rm DHO}-sets using orthogonal spreads}}]
\label{Constructing DHO-sets using orthogonal spreads}
 Example~\ref{CCKS}  contains the  construction
of the above set of operators using \cite[(7.4)]{CCKS} 
in terms of matrices
 (compare Remark~\ref{CCKS skew-symmetric construction}).
  However, the preceding theorem shows that we can proceed directly from  spread-sets to the required DHO-sets. 
 
 The examples studied in Sections~\ref{Sec 4} and \ref{Sec 5} 
 are obtained by taking known orthogonal spreads with ``nice'' descriptions
 in terms of matrices or linear operators and peeling off the set $\Delta$ in (\ref{Psi}).   Of course, there is a bias here:   orthogonal spreads having  nice descriptions  will  have less nice descriptions using arbitrary choices of its members $X,Y$  (as we will see in Example~\ref{transitive DHOs} below).

\begin{definition}[{{\em  Shadows}}]
\label{shadow}
Let   
$\Sigma $ be a spread-set of
self-adjoint operators  of $U$
 coordinatizing the symplectic spread ${\bf S}$
 of  $V=V(2n,2)$ with respect 
to the pair $(X,Y)$.
Let $Q$ be the unique quadratic form on $V$ polarizing
to the given symplectic form such that  $X$ and $Y$ are
totally singular.
 The DHO-set  $\Delta=\Delta _\Sigma$  associated to $\Sigma$ in 
    Proposition~\ref{Algebraic} will be called 
 the \emph{shadow of $\Sigma$}; it is uniquely determined by the spread-set. We also call the orthogonal
DHO on $(V,Q)$ defined by $\Delta$ a \emph{shadow} of 
the spread
${\bf S}$.
(Recall that this is not uniquely determined:  we choose $X$ and $Y$ in  order  to obtain the spread-set $\Sigma$ from the spread  ${\bf S}$. 
 Also see Section~\ref{The projections}.)
\end{definition}  
 
\begin{example} 
\label{Field} 
Consider $F=\FFF_{2^n}$ as an $\FFF_2$-space equipped with
the absolute  trace 
form  Tr   as a nondegenerate symmetric form.
Define  the  $\FFF _2$-linear map
$C(a), a\in F,$ by
$$
xC(a)= a^2x.
$$
Then $C$ is the canonical labeling (Definition~\ref{canonical})  of a symplectic spread-set
that coordinatizes the desarguesian plane.
The operators
$$
xB(a)= a^2 x+ {\rm Tr}(xa)a
$$
define the shadow
$\Delta =\{ B(a) \,| \, a \in F\}$ of $\Sigma$. 
In particular
$xE_{a,a}={\rm Tr}(xa)a$.
The automorphism group of  the  corresponding DHO is isomorphic to $F^\star \cdot {\rm Aut}(F)$
by Lemma~\ref{AutNearFlag} 
below.

\end{example}

Our later Examples~\ref{ShadowSemifield}  and 
\ref{ShadowNearFlag} 
are generalizations of this one.
We close this section
with a result obtaining new symplectic spreads from known ones.

\begin{theorem}
\label{Clone}

 Let 
$\Sigma$   be a  spread-set of self-adjoint operators of $U=V(n,q),$ and let
 $C=\LLL(\Sigma)$.
\begin{enumerate}
	\item[\rm(a)] 
	If  $u\in U,$ define $C_u \col  U\to {\rm End}(U)$ by
$$
C_u(a):= C(a)+E_{a,u}+E_{u,a}.
$$
Then 
$\Sigma _{u}:=\{ C_u(a)\,| \, a\in U\}$ is a spread-set of self-adjoint operators
and
 $C_u=\LLL(\Sigma_u)$.
 Moreover,  $\Sigma _{u}$ is additively closed if $\Sigma $ is.

\item[\rm(b)] Pick $1\neq \lambda \in \FFF _q$ and define
$C^\lambda \col U  \to {\rm End}(U)$ by
$$
C^\lambda(a):=C\big( (1+\lambda )a\big)+E_{\lambda a,\lambda a}.
$$
 Then 
	$\Sigma ^\lambda =\{ C^\lambda( a) \,| \,  a\in U\}$
	is a spread-set of self-adjoint operators and $C^\lambda=\LLL(\Sigma^\lambda)$.

\end{enumerate}
 
\end{theorem}

\begin{proof} 
This is a reformulation of
special cases of  \cite[Lemma 2.18]{KW}
using  Lemma \ref{SelfAdjoint}, 
  (\ref{Im Eab}) and Lemma~\ref{Labeling}(b).
(The easy,  direct algebraic verification -- similar to the proof of 
Theorem~\ref{Algebraic}
--
is left to the reader.)
\end{proof}

\begin{remark} 
In view of   \cite[Lemma 2.18]{KW},  
$\Sigma$, $\Sigma_u$   and $\Sigma^\lambda$  
are all projections of the same orthogonal spread
(cf.  Definition~\ref{Orthogonal and symplectic spreads}).

\end{remark}
\begin{definition}[{{\em Twists and dilations}}]
\label{Clones and shifts}

Let $\Sigma $  be
a symplectic spread-set of $U=V(n,q)$, $q$ even.
For $u\in U$ and $1\neq \lambda \in \FFF_q$
 we call the spread-set  $\Sigma _{u}$ in Theorem~\ref{Clone}(a)   the \emph{$u$-twist of
$\Sigma$}, and
the spread-set
 $\Sigma ^\lambda$ in Theorem~\ref{Clone}(b)
the \emph{$\lambda$-dilation of $\Sigma$}.
\end{definition}

\begin{corollary}
\label{Linear2}
In the notation of {\rm Theorem~\ref{Clone}(a),} assume that
 $q=2,$ $\Sigma$ is additively closed  and  $u\in U$.  Let $\Delta= \{ B(a):=C(a)+ E_{a,a} \,| \, a \in U\}$
 and $\Delta_u= \{ B_u(a):=C_u(a)+ E_{a,a} \,| \, a \in U\}$
 be the shadows of $\Sigma$ and $\Sigma_u$.  
 Then
$B_u(a)=B(a+u)+B(u)$.
\end{corollary}

\begin{proof} 
By Definition~\ref{shadow} and Theorem~\ref{Clone},   \begin{eqnarray*}
B_u(a) \hspace{-6pt} &=\hspace{-6pt}&C_u(a)+E_{a,a} \\
 \hspace{-6pt} &=\hspace{-6pt}& C(a) +E_{a,u} +E_{u,a}+E_{a,a} \\
\hspace{-6pt} &=\hspace{-6pt}& C(a+u) +E_{a+u,a+u}+C(u) +E_{u,u}  
 \\
\hspace{-6pt} &=\hspace{-6pt}& B(a+u) +B(u)   . \qedhere
\end{eqnarray*}
\end{proof}


\subsection{The projections ${\bf O}/N$ and ${\bf O}/P$}
\label{The projections} 
The term ``shadow'' of a symplectic spread suggests that, as in the physical world, 
the original object cannot be recovered from the shadow.  
We will see  how this occurs in our context:
the relationship between  symplectic spreads
and shadows is less tight
than visible in the preceding section.
 This is illustrated by Example~\ref{Connection}
below. 
We will see that  non-isomorphic spread-sets can produce isomorphic shadows,
a symplectic spread 
 can have non-isomorphic shadows,
and  the automorphism groups of a symplectic spread and a shadow
can be very different. These phenomena are best understood from the
viewpoint of orthogonal spreads:

\begin{proposition}
\label{Hyperbolic}
Let ${\bf O}$ be an orthogonal spread in $\overline{V}=V^+(2n+2,2)$.
Let $N$ be a nonsingular point  and $P$ a singular point in $\overline{V}$
such that the $2$-space $\langle N,P\rangle$ is hyperbolic.
Then the {\rm DHO} ${\bf O}/P$ is a shadow of the symplectic spread ${\bf O}/N$.

\end{proposition}

\begin{proof} We will use the notation
in Section~\ref{Projections and lifts with coordinates}
for a suitable choice  of coordinates. By assumption,
 $\langle N,P\rangle$ contains
a singular point $P' \neq P $.  
We may assume that  
$P'=\langle (e_1,0)\rangle$ and $P=\langle (0,e_1)\rangle$,
so that $N= \langle (e_1,e_1)\rangle $. 
We may assume that the members of  {\bf O}
containing 
$P'$ and $P$ are 
$X= \overline U\oplus 0 $ and $ Y= 0\oplus\overline U  $,
respectively.
According to Remark~\ref{Constructing DHO-sets using orthogonal spreads} (compare Example~\ref{CCKS}), 
${\bf O}/P$ is a shadow of ${\bf O}/N$.%
\end{proof}

\begin{example} 
\label{Connection}
(a)
When the usual desarguesian spread {\bf S} of $V(2,q^n)$
(for~$q$ even and $n>1$ odd) 
is viewed as a symplectic spread of $V(2n,q)$, it can be lifted to the {\em desarguesian orthogonal spread}  {\bf O}
of
$\overline V=V^+(2n+2,q)$ as~in 
Definition~\ref{Orthogonal and symplectic spreads}.
Then ${\bf O}/N_0={\bf S}$ for a nonsingular point $N_0$.
The group $G={\rm SL}(2,q^n) \cdot {\rm Aut}(\FFF_{q^n})$
preserves  the point $N_0$,  the orthogonal spread 
{\bf O} and  the   orthogonal geometry of $\overline V$.  It has exactly  two orbits of singular points; the various orbits of nonsingular points $N$ are described at length in \cite[Sec.~4]{Ka1}.
If $N\ne N_0$ then $\langle N^G\rangle$  is a $G$-invariant subspace  
$\ne 0, N_0$,
 and hence   is $N_0^\perp$ or $\overline V$.  

If $P$ is a singular point, then $P^\perp\ne N_0^\perp$.
Thus,  $P$ is not perpendicular to some member $N'$  of $ N^G $, in which case $\langle N', P\rangle $
is a hyperbolic 2-space.
\smallskip

(b) 
In particular, when $q=2$, by the preceding proposition {\em each ${\bf O}/P$ is isomorphic to a shadow of each 
${\bf O}/N$},  $N\ne N_0,$  where there are many non-isomorphic symplectic spreads ${\bf O}/N$
 \cite[Cor.~3.6 and Sec.~4]{Ka1}.
Also, ${\bf O}/P$ is a shadow of the desarguesian spread
${\bf O}/N_0={\bf S}$ when $P$ is not in $N_0^\perp$.

If $q=2$ and $n=5$,  then $G$ has precisely three orbits of
nonsingular points: $\{N_0\}$,  $N_1^G$, and $N_2^G$,
with $N_1^G\subseteq N_0^\perp$ and $N_2^G\cap N_0^\perp =\emptyset$.
Here
${\bf O}/N_1$ is a semifield spread with $|{\rm Aut}({\bf O}/N_1)|=2^5\cdot 5$,
and ${\bf O}/N_2$ is a flag-transitive spread with 
$|{\rm Aut}({\bf O}/N_2)|=33\cdot 5$.  The  two
orbits of $G$ on singular points are $P_0^{ G}$ (inside $  N_0^\perp$) and
$P_1^{ G}$ (with $P_1^{ G}\cap N_0^\perp =\emptyset$).
The  DHO  ${\bf O}/P_1$ appeared in Example~\ref{Field},  while
  ${\bf O}/P_0$  is one of the DHOs
in Example~\ref{transitive DHOs}.
By Example~\ref{Field},
  ${\rm Aut}({\bf O}/P_1)={ G}_{P_1} $  has  order $31\cdot 5$,
while 
 ${  G}_{P_0}$ induces on the DHO
${\bf O}/P_0$ an automorphism group of order $2^5\cdot 5$.
Thus, ${\bf O}/P_0\not\simeq {\bf O}/P_1$.
Use of a computer shows that ${\rm Aut}({\bf O}/P_0)={ G}_{P_0} $.

~

\end{example}



\section{Proof of Theorem~\ref{ShadowSemifieldKW}}
\label{Sec 4}
Except in Section~\ref{qDHOs}, we will use 
$F=\FFF _{2^n}$
with $n>1$   odd,  viewed as an $\FFF_2$-space  equipped
with the nondegenerate, symmetric bilinear form
$(x,y)\mapsto {\rm Tr}(xy)$ 
using  the absolute trace ${\rm Tr} \col  F\to \FFF_2$ 
as in Remark~\ref{Product}.

\Notation 
We 
will use the following$:$

\begin{itemize}
\item {}
The quadratic form
$Q$ on $V=F\oplus  F$ defined  by $Q(x,y)={\rm Tr}(xy);$
\item {}
  The trace map  ${\rm Tr}_{d:e} \col  \FFF_{2^d}\to \FFF_{2^e}$
  when $ \FFF_{2^d}\supset  \FFF_{2^e},$
  so that ${\rm Tr}_{n:1}={\rm Tr};$

	\item {}  Sequences $\underline{d}=(d_0=1,d_1, \ldots , d_m)$	of  $| \underline{d} |=m+1$ different integers such that 
	 $ d_1| d_2 | \cdots | d_m | n$, 
	associated with a chain
	 $\FFF_2 =F_0\subset F_1 \subset \cdots \subset F_m\subset F,$
	 $|F_i| ={2 ^{d_i}}$,
	 of     $| \underline{d} |  $  proper   subfields of $F;$    
	\item   
	The   $F_i$-linear operator
\begin{equation}
\label{Eab}
E_{a,b}^{(i)}\colon x\mapsto {\rm Tr}_{n:d_i}(ax)b
\end{equation}
	on $F$  for $a,b\in F$ and   $0\leq i \leq m;$	and
 	\item  Sequences $\underline{c}=(c_1, \ldots ,c_m), c_i\in F$.
\end{itemize}

This section  is concerned with
 the following  symplectic semifield spread-sets:
\begin{example}
\cite{KW}\,
\label{ShadowSemifield}
Let $\underline{d} $ and $\underline{c}$
be as above
with  all $c_i\in F^\star$.
For $a\in F$ define the operator
$C(a)$ on $F$ by
$$
C(a)= a^2{\bf 1}+ \sum _{i=1}^{m} (E_{c_i,a}^{(i)}+E_{a,c_i}^{(i)}).
$$
This defines a symplectic spread-set $\Sigma$. 
Moreover, $C=\LLL(\Sigma)$ by Example~\ref{Field}
since the operators $E_{c_i,a}^{(i)}+E_{a,c_i}^{(i)}$ are skew-symmetric.
The shadow of $\Sigma$ (Definition~\ref{shadow}) is 
$$
\mbox{$\Delta =\{B(a) \,| \, a\in F\}$ \ with \
$B(a)=C(a)+ E_{a,a}^{(0)}$.}
$$
The DHO-set $\Delta$ defines an orthogonal DHO of
$V$ by
$$\mbox{
${\bf D}\,=\, \{V(a)\,| \, a\in F\}$ 
\  with \ $ V(a)=V\big(B(a)\big):=\{ \big(x,xB(a)\big)\,| \, a\in F\}.$ }
$$

\end{example}

\begin{remark}
\label{ExampleCloning}
(a)
The preceding spread-set $\Sigma$ is obtained by successively twisting  the
desarguesian spread-set $\Sigma _0=\{ a^2{\bf 1} \,| \, a\in F\}$.
Namely,      view
 $\Sigma _0$ as a symplectic spread-set over $F_m$.
Let $d=d_m$ and $ c=c_m\in F^\star$.
By Theorem~\ref{Clone} the twist
 $\Sigma _1=\{ a^2{\bf 1}+ E_{c,a}^{(m)}+E_{a,c}^{(m)} \,| \, a\in F\}$
  is a symplectic spread-set over $F_m$.
Now view   $\Sigma _1$ as a spread-set over $F_{m-1}$
 and iterate the twisting using $c_{m-1}\in F^\star$.
 
 (b)  None of the nontrivial  elations
 of the projective plane arising  from the symplectic
spread-set is inherited by the shadow DHO since 
 $C(a+b)=C(a)+C(b)$ but $B(a+b)=C(a+b)+ E_{a+b,a+b}^{(0)}
\neq C(a)+  E_{a,a}^{(0)} +C(b) + E_{b,b}^{(0)} =B(a)+B(b)$ for $0\neq a,b$,
$a\neq b$. 
 
\end{remark}

Our goal is to   show that  we obtain
at least
   $2^{n(\rho (n)-2)}/n^2$ inequivalent orthogonal DHOs of the above type
when  $n$ is composite.
We start with a uniqueness result concerning  shadows:  

\begin{proposition}
\label{Linear1}
If $n>5$,
then a  \DHO\!\!-set   $\Delta 
\subseteq {\rm End}(U)$  can be the shadow
of at most one additively closed symplectic spread-set.

\end{proposition} 

\begin{proof} Let $\Delta =\{ B(a) \,| \, a\in U\}$
be the shadow of the  additively closed symplectic spread-sets 
$\Sigma$ and $\widetilde{\Sigma}$.
Write 
$\Sigma 
=\{ C(a):=B(a)+E_{a,a} \,| \, a\in U\}$
with $C=\LLL(\Sigma)$ additive (by Lemma~\ref{Labeling}(b)).
Then for each $B(a)\in \Delta$ there is a self-adjoint operator $E_{b,b}$ of rank $\leq 1$  such that $\overline{C}(a): =B(a)+ E_{b,b}\in \widetilde{\Sigma}$. Write $a'=b$.
We have to show that $a'=a$ for all $a$.
(N.\,B.--We do not know that $\overline{C}=\LLL(\widetilde\Sigma) $.)
 
We claim that $\overline C$ is additive.  
Let $0\neq a,b\in U$ and  $\overline{C}(a)+\overline{C}(b)=\overline{C}(c)$  with $c\in U$.
By the additivity of  $C$ and the definition of $\overline C$,
\begin{equation}
\label{C()}
\begin{array}{lllll}
C(a+b+c)\hspace{-6pt}&=\hspace{-6pt}& 
(B(a)+E_{a.a})+(B(b)+E_{b.b})+(B(c)+E_{c.c})
\vspace{2pt}
 \\
\hspace{-6pt}&=\hspace{-6pt}& 
E_{a',a'}+E_{b',b'}+E_{c',c'} + E_{a,a}+  E_{b,b}+  E_{c,c} .\end{array}
\end{equation}
Then $c=a+b$, as otherwise the rank of the above left side is $n$ and of the right side is
$\leq 6$.
Thus,  $\overline{C}$ is additive. 

Since
$\overline{C}(a) + E_{a',a'}$ and $\overline{C}(b) + E_{b',b'}$
are skew-symmetric, by  Lemma~\ref{SelfAdjoint}(c) 
$$
\overline{C}(a) + \overline{C}(b) + E_{a'+b',a'+b'}
=\overline{C}(a+b)  + E_{a'+b',a'+b'}
$$ 
is also skew-symmetric.
Since $\overline{C}(a+b)  + E_{(a+b)',(a+b)'}$ is skew-symmetric,
 $a'+b'=(a+b)'$  by   Lemma~\ref{SelfAdjoint}.
 
 Since $a+b=c$ and $E_{a+b,a+b}=E_{a.a}+E_{b,b}+E_{a,b}+E_{b,a}$, we have 
$
 E_{a,b}+  E_{b,a}=E_{a',b'}+E_{b',a'}
$
by  (\ref{C()}).
By (\ref{Im Eab}), 
$$
\langle a',b' \rangle ={\rm Im}\, (E_{a',b'}+E_{b',a'})=
{\rm Im}\, (E_{a,b}+E_{b,a})
 =\langle a,b \rangle .
$$
 Then the additive map $a\mapsto a'$ fixes each
 $2$-space of the $\FFF_2$-space $U$, and hence is 1.
\end{proof}

Theorem~\ref{ShadowSemifieldKW} 
depends on relating equivalences of   spread-sets and 
  of   shadows of twists (cf. Definition~\ref{Clones and shifts}):

\begin{theorem} 
\label{IsoShadow}
Assume that 
 $\Sigma $
and $\widetilde{\Sigma} $ are symplectic spread-sets
of   $U = V(n,2),$  for odd $n > 5,$
whose respective shadows $\Delta$
 and 
$ \widetilde\Delta$ 
are equivalent.

\begin{enumerate}
	\item[\rm (a)] For some permutation $a\mapsto a'$ of $U$
	fixing $0,$ some  
	$T\in {\rm GL}(U)$ and some  $u\in U,\,$
	$T^\star B(a)T=\widetilde{B}_u(a')$ for all $a\in U,$ 
	where  $\Delta =\{ B(a) \,| \,a\in U\}$ is the shadow of $\Sigma$
	    and 
	 $\widetilde{\Delta}_u
	 	=\{ \widetilde{B}_u(a) \,| \,a\in U\}$
	is   the shadow of the twist   $\widetilde\Sigma_u$.
	
		\item[\rm (b)] If $\widetilde\Sigma$ is additively closed
		  then$,$  for  some permutation $a\mapsto \overline{a}$ of $U$
	and    $S=T^{-1} ,$
  $\widehat{C}(a): = B(\overline{a})+E_{a,a}
	=
S^\star \widetilde{C}_u(aT) S$ 
	is the canonical labeling of the additively closed symplectic spread-set  
	$S^\star \widetilde{\Sigma}_u(a) S$.
	Furthermore$,$ 
	 the shadow of $\widehat{\Sigma}=S^\star \widetilde{\Sigma}_uS$
	is $\Delta$.

	\item[\rm (c)] It $\Sigma$ and $\widetilde\Sigma$ are additively closed  then a semifield defined by $\Sigma$ is isotopic to a semifield defined by some twist of $\widetilde\Sigma$.
\end{enumerate}
\end{theorem}  
 
 See \cite[p.~135]{Demb} 
for the definition of  ``isotopic semifields''.
In the present context this means that 
\emph{$T_1\Sigma  T_2$  is a twist of   $\widetilde  \Sigma$ for some   $T_1,T_2\in  {\rm GL}(U)$.}

\begin{proof}  (a)
Let $\Phi  \col V \to V$ be an operator  mapping the DHO
${\bf D}$ for $ \Delta$ onto the DHO $\widetilde{\bf D}$  for  
$\widetilde{\Delta }$,
where  $V=U\oplus  U$  as usual.   By  Proposition~\ref{Basic}(b)  and 
Proposition~\ref{Focus}, $\Phi\in O(V)$ has the form
$$
(x,y)\Phi = (x\Phi_{11}, x\Phi _{12}+y\Phi_{22})
$$
where  $\Phi _{11},\Phi_{22}\in {\rm GL}(U)$, 
$\Phi _{12}\in {\rm End}(U)$,  and 
the adjoint of 
$T:= \Phi_{22}$ is  
 $T^\star=\Phi_{11}^{-1}$   by
 Corollary~\ref{Contragredient}.

 If  $C=\LLL(\Sigma)$ and $\widetilde C=\LLL({\widetilde\Sigma})$
(Definition~\ref{canonical}),
we have  $\Delta =\{ B(a):=C(a)+E_{a,a}\,| \, a\in V\}$
and $\widetilde\Delta  
=\{ \widetilde{B}(a):=\widetilde{C}(a)+E_{a,a}\,| \, a\in U\}$.
Then  ${\bf D}= \{V({B}(a))\,| \, a\in U \}$
 and $\widetilde{\bf D}= \{ V(\widetilde{B}(a))\,| \, a\in U \}$
 in the notation of (\ref{V}).
 
We apply $\Phi$ to $(x,xB(a))\in V({B}(a))\in {\bf D}$ and obtain 
$$
(x,xB(a))\Phi =(y, y\Phi _{11}^{-1}(\Phi _{12}+B(a)\Phi _{22} ) ) \in  V(\widetilde{B}(a')),
\quad y= x\Phi _{11},
$$
for some permutation $a\mapsto a'$ of $U$. 
Then 
$\widetilde{B}(a')=T^\star(\Phi _{12}+B(a)T) $. In particular,
when $a=0$ and   $u:=0'$ we have $\widetilde{B}(u)=T^\star\Phi _{12}$.
Then, in the notation of Corollary~\ref{Linear2},
  $T^\star B(a)T=\widetilde{B}(a')+\widetilde{B}(u)
 =\widetilde{B}((a'+u)+u)+\widetilde{B}(u) =\widetilde{B}_u(a'+u)$.  
Since $0'=u$, replacing  $a\mapsto a'$
 by the permutation  $a\mapsto a'+u$   produces  (a)
 (but does not change $u$). 
 
 \smallskip
(b)  If $\widetilde\Sigma$
is additively closed  then  $\widetilde{C}_u$ is additive  by
Lemma~\ref{Labeling}(b) and  the end of 
Theorem~\ref{Clone}(a).

Let $a\mapsto \overline{a}$ be the inverse
of $a\mapsto a'S$.
Then   (a) states that
$
\widehat{C}(a)=
B(\overline{a})+E_{a,a}=
S^\star \widetilde{B}_u(\overline a') S  +E_{a,a}
=
S^\star \widetilde{C}_u(aT) S +S^\star  E_{aT,aT}S+E_{a,a}
=
S^\star \widetilde{C}_u(aT) S 
$.   
The shadow of the symplectic spread-set $\widehat \Sigma$  for $\widehat{C}$ is $
\{ {B}(\overline a)\,|\, \overline a\in U\} =\Delta$, by  Definition~\ref{shadow}, while   
	$\widehat \Sigma = S^\star  \widetilde \Sigma _u S$.
Finally,  the additivity of  $a\mapsto S^\star \widetilde{C}_u(a) S$ proves (b).

 \smallskip
(c) This is immediate from
 (b) and Proposition~\ref{Linear1}.%
\end{proof}

\noindent
{\sc Proof of Theorem~\ref{ShadowSemifieldKW}}. 
By \cite[Thm. 4.15]{KW}, \cite[Thm. 1.1]{Ka2} and \cite{LMPT}, there
are at least $2^{n(\rho (n)-1)}/n^2$ symplectic semifield
spreads defining non-isomorphic semifield planes using
Example~\ref{ShadowSemifield}. 
If two equivalent  orthogonal DHOs are defined by  shadows of symplectic spread-sets $\Sigma $ and $\widetilde\Sigma $
 in Example~\ref{ShadowSemifield},  then 
 the semifields
defined by $\Sigma $ and some twist $\widetilde{\Sigma}_{ u}$ ($u\in U$)
are isotopic by
Theorem~\ref{IsoShadow}(c). 
Since there are  $|U|=2^n$ possibilities for $u$, 
we obtain at least 
$2^{n(\rho (n)-2)}/n^2$ pairwise   inequivalent DHOs.  
\qed

\remark
Note that the exact formulas for the semifield spreads
in Example~\ref{ShadowSemifield}
 were never used in the above arguments.
Therefore, if   many more inequivalent symplectic semifield spread-sets are found then there will, correspondingly, be many more inequivalent DHOs.

Also note that Proposition~\ref{Linear1} and Theorem~\ref{IsoShadow}
  deal with spread-sets and DHO-sets, and hence do not conflict with Section~\ref{The projections}, which deals with spreads and DHOs.

The preceding result and argument differ in a significant way from ones 
in 
\cite{DM,KW,KW1}  and Section~{\ref{Sec 5}}:  it did not rely on a group of automorphisms of the 
objects (DHOs) being studied, but rather on such a group for
related objects.



\section{Proof of Theorem~\ref{ShadowNearlyFlagKW}}
\label{Sec 5} 
We will show  that  the shadows of the symplectic spreads
of the nearly flag-transitive planes in \cite{KW1}
produce  at least as many
non-isomorphic DHOs as stated in Theorem~\ref{ShadowNearlyFlagKW}. 
We start with the corresponding spread-sets:

\begin{example}
 \cite{KW1}\,
\label{ShadowNearFlag}
Let $\underline{d} $ and $\underline{c}$ be sequences  as 
at the start of  the
preceding section, with associated fields $F_j$ and
 the additional properties that 
 $c_j\in F_j$ 
 with at least one of them nonzero,  and
$\sum _{i=1}^j c_i\neq 1$ for $1\leq j\leq m$.
For $a\in F$ define  
\begin{equation}
\label{DHO C}
C(a)= (1+\sum _{i=1}^{m} c_i)a^2{\bf 1}+ \sum _{i=1}^{m} c_iE_{a,a}^{(i)} 
\end{equation}
(the operators $E_{a,b}^{(i)}$ are  in
(\ref{Eab})).
Then $C$ is the canonical labeling  
 of a symplectic spread
set $\Sigma$. 
Indeed, $\Sigma$ is just the  
description in  \cite{DM} of the symplectic spread-sets from 
\cite{KW1}. 
For completeness we verify 
that $C$ is   the canonical labeling  $\LLL(\Sigma)$, i.\,e., in view of 
(\ref{Im Eab})  and  Definition~\ref{canonical},
that ${\rm Tr}\big(x (xC(a))\big)={\rm Tr}\big(x (xE_{a,a})\big)={\rm Tr}(ax)^2$
 (as in \cite[(2.15)]{KW}).
Since  $n$ is odd  we have ${\rm Tr}={\rm Tr}\circ {\rm Tr}_{n:d_i}$ and hence 
${\rm Tr} \big(c_iz {\rm Tr}_{n:d_i}(z) \big)=
{\rm Tr}\circ{\rm Tr}_{n:d_i} \big(c_iz {\rm Tr}_{n:d_i}(z) \big)=
{\rm Tr} \big(c_i    {\rm Tr}_{n:d_i}(z)^2 \big)=
{\rm Tr} \big(   {\rm Tr}_{n:d_i}(c_i  z^2) \big)=
{\rm Tr} (c_i  z^2)  $. 
   If $z=ax$ it follows that
$$
{\rm Tr}\big(x (xC(a)\big)={\rm Tr}(ax)^2 +\sum _{i=1}^m{\rm Tr}(c_ia^2 x^2)
+\sum _{i=1}^m{\rm Tr} \big(c_iax {\rm Tr}_{n:d_i}(ax) \big)
={\rm Tr}(ax)^2,
$$
as required.

The shadow of $\Sigma$ is 
\begin{equation}
\label{DHO B}
\mbox{
$\Delta =\{B(a) \,| \, a\in F\}$ \ with \ 
$B(a)=C(a)+ E_{a,a}^{(0)}.$}
\end{equation}
Using the quadratic form in the preceding section,
we obtain a DHO in 
$F\oplus F$:
$$\mbox{
${\bf D}={\bf D}_{\underline{d},\underline{c}}:=\{ V(a) \,| \,  a\in F\}$  \  with \ $ V(a) : =V\big(B(a)\big)$.}
$$
For $b\in F^\star$ define $M_b\in {\rm GL}_{\FFF_2}(F)$
by $(x,y)M_b=(b^{-1}x,by)$. If $y=b^{-1}x$  then
$$
\big(x,xB(a)\big)M_b=\big(y,b(byB(a))\big)=\big(y,yB(ab)\big),
$$
so that $V(a)M_b=V(ab)$ 
 in the notation of (\ref{V}),
 and ${\cal M } :=\{ M_b\,| \, b\in F^\star \}\simeq F^\star$ is a group of automorphisms of ${\bf D}$.
 Also, if $\alpha \in {\rm Aut}(F)$
 then the map 
 \begin{equation}
 \label{Phi alpha}
 \Phi _\alpha \col (x,y)\mapsto (x^\alpha , y^\alpha )
  \end{equation}
normalizes ${\cal M}$  and it is an
automorphism of 
 of ${\bf D}$ if $c_i^\alpha= c_i$  for all $i$.
Define 
 \begin{equation}
 \label{P}
  \mbox{${\cal P}=
 \{\Phi _\alpha \,| \,  c_i^\alpha= c_i$  for all $i\}$
\  and  \   $
 {\cal G}={\cal M}{\cal P}.$}
 \end{equation}
In the next proposition we will show that ${\cal G}$ is 
the full automorphism group of  ${\bf D}$.
\end{example}

\begin{remark} 
\label{ExampleLifting}
(a)
The preceding spread-set $\Sigma$ 
is obtained by successively dilating   the
desarguesian spread-set $\Sigma _0=\{ a^2{\bf 1} \,| \, a\in F\}$.
View  $\Sigma _0$ as a symplectic spread-set over $F_m$.
Let $d=d_m$ and $1\neq c=c_m\in F_m^\star$, and define $\lambda=c^{1/2}$.
By Theorem~\ref{Clone}, a typical element of the $\lambda$-dilation
has the form 
 $ ((1+\lambda)a)^2{\bf 1}+ E_{\lambda a,\lambda a}^{(m)}
 =(1+c)a^2{\bf 1}+ c E_{ a,a}^{(m)}$, where the right side
   is  $C(a)$ when $m=1$.
   Hence the spread-set $\Sigma$ is obtained as a dilation 
     in the case $m=1$.
 View $\Sigma$  as a spread-set over $F_{m-1}$
 and  iterate the dilating by choosing  $c_{m-1}\in F_{m-1}$.%
 \smallskip

 (b)
Two DHOs
${\bf D}_{\underline{d},\underline{c}}$
and ${\bf D}_{\underline{d'},\underline{c'}}$
are \emph{equal if and only if  $\underline{d}=\underline{d'}$ and
$\underline{c}=\underline{c'}$.}  This   is  proved
exactly as in \cite[Prop. 8.1]{KW1} or \cite[Proof of Thm. 5.2]{DM}. \smallskip

 (c) When $m=0$  Examples~\ref{ShadowSemifield}
and  \ref{ShadowNearFlag}  coincide with  Example~\ref{Field}.

\smallskip

(d)  
Unfortunately,   use of  Theorem~\ref{IsoShadow}(a)  
does not seem to shorten the proofs in the present section.
\end{remark}

\begin{proposition}
\label{AutIso}
Let ${\bf D}={\bf D}_{\underline{d},\underline{c}}$
and
${\bf D'}={\bf D}_{\underline{d'},\underline{c'}}$
be  {\rm DHO}s  in {\rm Example~\ref{ShadowNearFlag}. }
Then  
\begin{enumerate}
	\item[\rm (a)] ${\rm Aut}({\bf D})={\cal G},$  and
	 
	\item[\rm(b)]  ${\bf D}\simeq {\bf D}'$ if and only if 
	$\underline{d}=\underline{d'}$ and  
	 $c_i^\alpha =c_i'$ for some  $\alpha \in {\rm Aut}(F)$ and
	 $1\leq i \leq |\underline{d}|$.

\end{enumerate}

\end{proposition}

 We will prove this  using several  lemmas.
Recall
 that ${\bf D}$ and ${\bf D}'$ split  over $Y=0\oplus F\subseteq V$.
 \begin{lemma}
\label{NoTranslation}
If ${\Phi\in {\rm Aut}({\bf D})}$ satisfies   ${\Phi _Y={\bf 1}_Y}$
and ${\Phi _{V/Y}={\bf 1}_{V/Y}}$ then ${\Phi ={\bf 1}}$. 

\end{lemma}

\begin{proof}
By assumption, $(x,y)\Phi=(x,xR+y)$ for some $R\in {\rm End}(F)$.
There is a permutation $a\mapsto a'$ of $F$ such that 
$V(a)\Phi=V(a')$ for all $a$.  Then $(x,xB(a))\Phi=(x,xB(a'))$ states that
$R+B(a)=B(a')$ for all $a$.
Let $b:=0'$, so that   $R=B(b)$.

If $b=0$ then $\Phi ={\bf 1}$, as required.

Suppose that  $b\neq 0$.  We have $B(a)+B(b)=B(a')$.
Consider the  equation $x  B(a)+xB(b)=xB(a')$
as a polynomial equation modulo $x^{2^n}-x$.
By (\ref{DHO C}) and (\ref{DHO B}), $xB(a) $ is the sum of a term 
linear in $x$, terms of the form $cx^{2^{d_ik}}$ with 
$d_i > 2$ and $0< d_ik<n$
 and $c\in F$, and 
terms such as  $a^{1+2^{k}}x^{2^{k}}$ arising from ${\rm Tr}(ax)a$.
If   $0<k<n$, $(k,n)=1$, then  
$$
a{}^{2^k+1}x^{2^k}+b^{2^k+1}x^{2^k}=a'{}^{2^k+1}x^{2^k}, \quad x\in F, \quad {\rm i.\,e.,} \quad a^{2^k+1}+b^{2^k+1}=a'{}^{2^k+1}.
$$
Choosing  $k=1$ and $k=2$, since  $(a'{}^3)^5=(a'{}^5)^3$  we see that
 every $x\in F$
satisfies
$(x^3+b^3)^5=(x^5+b^5)^3,$
 which is absurd since $b\ne0$.
\end{proof}

\begin{lemma}
\label{Singer}
${\rm Aut}({\bf D})$ is isomorphic to
a subgroup of $\Gamma {\rm L}(1,2^n),$ and ${\cal M}$ is normal in
${\rm Aut}({\bf D})$.
\end{lemma}

\begin{proof} Set ${\cal A}  :={\rm Aut}({\bf D})$.
By Lemma~\ref{NoTranslation} and Corollary~\ref{Contragredient},
${\cal A}$ acts faithfully  on $Y$, and  ${\cal M}$ induces
a Singer group of ${\rm GL}(Y)$.
By  \cite{Ka0},   ${\cal A}$ has a normal
subgroup  ${\cal H}\simeq {\rm GL}(k ,2^\ell )$,  where  $n=k\ell$
and  $ {\cal Z}:={\cal M}\cap Z({\cal H})$ is
a cyclic group of order $2^\ell -1$. 
If $k=1$, then ${\cal H}={\cal M}$, as required.

Assume   that $k>1$. The   ${\cal M}$-orbits  on ${\bf D}$
are $\{V(0)\}$ and ${\bf D} - \{V(0)\}$.
Then $V(0)$ is ${\cal H}$-invariant,
as otherwise  ${\cal H}$  would be 2-transitive on ${\bf D}$, contradicting \cite{CKS}. 
The action of ${\cal M}$ on $V(0)$ is the same as its action on the field $F$, hence   $V(0)$
can be viewed as an $\FFF_{2^\ell}$-space on which ${\cal Z}$ acts  as 
 $\FFF_{2^\ell}^\star$  and  $ {\cal H}  $   acts as  ${\rm GL}(k ,2^\ell )$.

 In order to obtain a contradiction 
 we will use a transvection  $A$ in ${\rm GL}(k ,2^\ell )$  (so that  the $\FFF_2$-space $W:=C_{V(0)}(A)$ has dimension  $n-l$  and $A^2=1$;
 from now on 
dimensions will   be over $\FFF_2$).  By Corollary~\ref{Contragredient},  $A$ arises from  an operator $\Phi\in {\cal H}$
such that
  $(x,y)\Phi=(xA, y(A^\star)^{-1} ) = (xA, yA^\star)$.

There is a permutation $a\mapsto a'$ of $F^\star$ such that 
$V(a)\Phi =V(a')$. 
Then  $AB(a)A ^\star = B(a')$ since 
$V(a)\Phi =\{ (x, xAB(a)A ^\star ) \,| \,
x\in F\}$.

Note that $W(AB(a)A^\star+B(a))\subseteq
WB(a)(A^\star+ {\bf 1})$ has dimension 
$\le {\rm rk}(A^\star+ {\bf 1})=l$.
Since $\dim V - \dim W =n-(n-l)$,  it follows that
  ${\rm rk}(B(a')+B(a))=\dim V(0)(AB(a)A^\star+B(a))\le l+l<n-1$.
By Lemma~\ref{Coord}(c),
  $a'=a$ and hence  $\Phi ={\bf 1}$,
a contradiction.%
\end{proof}

\begin{lemma}
\label{AutNearFlag}
 ${\rm Aut}({\bf D})={\cal   G}$. 

\end{lemma}

\begin{proof}
By the preceding lemma,  
we need to determine which  
  $\Phi _\alpha $ lie in ${\cal G}$. 
Since
$V(a)\Phi _\alpha =\{ (x^\alpha , (xB(a) )^\alpha ) \,| \,x\in F\}$,
 (\ref{DHO C}) and (\ref{DHO B}) show that $V(a)\Phi _\alpha =V(a^\alpha) $, so that 
  ${\bf D}_{\underline{d},\underline{c}}=
{\bf D}_{\underline{d},\underline{c}^\alpha}$.
By Remark~\ref{ExampleLifting}(b), $c_i=c_i^\alpha$ for all $i$,   so that
$\Phi _\alpha \in {\cal P}$.%
\end{proof}

\begin{remark}
 It might be interesting to have a proof of 
 Lemma~\ref{AutNearFlag}
using an elementary polynomial argument rather than the somewhat less elementary group
theory we employed.
\end{remark}

\noindent
\emph{Proof of} Proposition~\ref{AutIso}.
We   just proved    (a). 
Consider   (b).
Clearly,  $\Phi _\alpha$ 
maps ${\bf D}_{\underline{d},\underline{c}}$ onto
${\bf D}_{\underline{d},\underline{c}^\alpha }$ (cf. (\ref{Phi alpha})).

Conversely, assume that $\Phi$ maps ${\bf D}$ onto
${\bf D}'$. By Proposition~\ref{Focus}, $\Phi$ lies in ${\rm O}(V)$,
and by Lemma~\ref{Singer} it  even
 lies in the normalizer 
 ${\cal M}\{ \Phi _\alpha \,| \, \alpha \in {\rm Aut}(F)\}$
 of
${\cal M}$ in ${\rm O}(V)$.
(Compare the proofs of \cite[Prop. 5.1]{KW1} or
\cite[Prop. 4.6]{DM}; the former does not even need the 
precise group   ${\rm Aut}({\bf D}) $.) So we may assume
that  $\Phi =\Phi _\alpha$
for some  $\alpha$. Arguing
as in the proof of the preceding lemma we
obtain $\underline{d}=\underline{d}'$ and
$\underline{c'}=\underline{c}^\alpha$.
\qed
\medskip

We leave the following calculation to the reader:
\begin{lemma}
\label{Estimate} 
If $p_1 \leq \cdots \leq p_\ell$ are odd primes$,$   then
$$
\frac{(2^{p_1}-1)(2^{p_1p_2}-1)\cdots (2^{p_1 \cdots p_\ell}-1)}{p_1 \cdots p_\ell} \geq 2^{3 ^\ell}
$$
unless $(\ell ; p_1, \ldots, p_\ell )= (1;3), (1;5)$ or $(2;3,3)$.

\end{lemma}

\noindent
{\sc   Proof of Theorem~\ref{ShadowNearlyFlagKW}.}
Let $n=p_1p_2\cdots  p_{m+1}$ for odd primes $p_i$
such that 
$p_1 \leq \cdots \leq p_{m+1}$,
i.e. $\rho (n)=m+1$. 
Consider the chain $\FFF _2 =F_0 \subset F_1 \subset F_2 \subset \cdots \subset F_{m+1} =F=\FFF _{2^n}$
where $|F_i|= {2^{d_i}}$ 
for $d_i=p_1 \cdots p_i$.
Every sequence 
$(c_1, \ldots ,c_m)$  with  $c_i\in F_i$ 
and
$\sum _{i=1}^j c_i\neq 1$ for $1\leq j\leq m$
defines 
a symplectic spread  in 
Example~\ref{ShadowNearFlag}
 (where $c_i=0$ means that we delete the field $F_i$
from the chain).  By Proposition~\ref{AutIso} we obtain at least 
$(2^{p_1}-1)(2^{p_1p_2}-1)\cdots (2^{p_1 \cdots p_m}-1)/p_1p_2\cdots  p_m$
pairwise inequivalent DHOs. 
Now use
Lemma~\ref{Estimate}.
\qed

\section{\bf A non-isomorphism theorem}
\label{NoIso}

In this section  we will prove:

\begin{theorem}
\label{NotIsomorphic}
Any \DHO from {\rm Example~\ref{ShadowSemifield}}
is not isomorphic to a \DHO from {\rm Example~\ref{ShadowNearFlag}}
having $h>0$.

\end{theorem}

First we need a tedious computational result:  
 \begin{lemma}
 \label{Technical}
   For $F=\FFF_{2^n}$ $(n\geq 5$ odd$),$
   let $f\col F\to F $   be such that   $f(x)^3+x^3$ 
 and $f(x)^5+x^5$  are additive. 
 Then $f=1$.
 
  \end{lemma}
 \begin{proof}
Let 
 $g(x):=f(x)^3+x^3$ and    $h(x):=f(x)^5-x^5$.
Since  
 $(f(x)^3)^5=(f(x)^5)^3$,  for all $x\in F$ we have
 \begin{equation}
\label{Quadratic}
x^{12}g(x)+x^3g(x)^4+g(x)^5  = x^{10}h(x)+x^5h(x)^2 +h(x)^3 .
\end{equation} 
 
 Write  $g(x)=\sum _{i=0}^{n-1}g_ix^{2^i}$
 and $h(x)=\sum _{i=0}^{n-1}h_ix^{2^i}$
 with  $g_i,h_i\in F$, where    indices will be  read   mod
$n$.
 Since $h(x)^2=\sum _{i=0}^{n-1}h_{i-1}^2x^{2^i}$
 and $g(x)^4=\sum _{i=0}^{n-1}g_{i-2}^4x^{2^i}$,
  the left side of (\ref{Quadratic})  has the form
$$
L(x)=\sum _{i=0}^{n-1}g_ix^{2^i+12}+ \sum _{i=0}^{n-1}g_{i-2}^4x^{2^i+3}
+g(x)^5
$$
and the right side has  the form
$$
R(x)= \sum _{i=0}^{n-1}h_ix^{2^i+10}
+\sum _{i=0}^{n-1}h_{i-1}^2x^{2^i +5}+h(x)^3.
$$
In order to view $L(x)=R(x) $ as a polynomial identity involving polynomials of degree $\le2^n-1$,   we note that the above  summations in $L(x)$ and $R(x)$ involve exponents $\le2^n-1$ (since $n\ge5$),
as  do the following (for all $x\in F$): 
\begin{eqnarray*}
g(x)^5\hspace{-6pt} &=\hspace{-6pt} &\left( \sum _{i=0}^{n-1}g_ix^{2^i} \right)
\left( \sum _{i=0}^{n-1}g_{i-2}^4x^{2^i} \right)\\
\hspace{-6pt} &=\hspace{-6pt} &
\sum _{0\le i<k \le n-1}(g_ig_{k-2}^4 +g_kg_{i-2}^4)x^{2^i+2^k}
+\sum _{i=0}^{n-2}g_ig_{i-2}^4x^{2^{i+1}}    +g_{n-1}g_{n-3}^4x
\\
h(x)^3\hspace{-6pt} &=\hspace{-6pt} &\left( \sum _{i=0}^{n-1}h_ix^{2^i} \right)
\left(\sum _{i=0}^{n-1}h_{i-1}^2x^{2^i} \right)\\
\hspace{-6pt} &=\hspace{-6pt} &
\sum _{0\le i<k \le n-1}(h_ih_{k-1}^2 +h_kh_{i-1}^2)x^{2^i+2^k}
+\sum _{i=0}^{n-2}h_ih_{i-1}^2x^{2^{i+1}}    +h_{n-1}h_{n-2}^2x.
\end{eqnarray*}

Denote by $L_o(x)$ and $R_o(x)$ the sums over the terms
with  odd  exponents
 in $L(x)$
and $R(x)$, respectively.  These    involve the following  exponents:
$$\begin{array}{llllll}
L_o(x) &2^0+12  \ &2^i+3 \  (i>0)&  \hspace{4pt}1 \hspace{4pt}  &  2^0+2^k \ (k>0)
\\
R_o(x) &2^0+10 \ &2^i+5 \  (i>0)& \hspace{4pt} 1  \hspace{4pt} &  2^0+2^k \  (k>0)
\end{array}
$$
We   rewrite $L_o(x)$ and $R_o(x)$ so that all coinciding exponents are visible: 
\begin{eqnarray*}
L_o(x)\hspace{-6pt} &=\hspace{-6pt} &g_{n-1}g_{n-3}^4x+ (g_{-1}^4+ g_{0}^5 +g_2g_{-2}^4)x^5+
g_0x^{13}
 \\
& &
\hspace{18pt}+ [g_{0}^4x^7 +g_1^4x^{11}] + (g_0g_{1}^4 +g_3g_{-2}^4)x^9  \\ 
& &  
+  \sum _{i\geq 4} g_{i-2}^4x^{2^i+3}
+ \sum _{\substack{0<k\le n-1 \\ k\neq 2 ,3}} 
(g_0g_{k-2}^4 +g_kg_{-2}^4)x^{1+2^k} 
\\R_o(x)\hspace{-6pt} &=\hspace{-6pt} &
h_{n-1}h_{n-2}^2x  + (h_{1}^2+ h_0h_{2}^2 +h_3h_{-1}^2)x^9+
h_0x^{11}
 \\
& &
\hspace{18pt}+[ h_0 ^2 x^{7} + h_2 ^2x^{13} ]+  (h_0h_{1}^2 +h_2 h_{ -1}^2)x^5  \\
& &
+  \sum _{i\geq 4}h_{i-1}^2x^{2^i+5} +
\sum _{\substack{0<k\le n-1 \\ k\neq 2,3}}
 (h_0h_{k-1}^2 +h_k h_{-1}^2)x^{1+2^k}.
\end{eqnarray*}
  Comparing
  the coefficients of $L_o(x)=R_o(x)$, 
we obtain  the following table containing some of the relations  among the various $g_i$ and $h_i$.

\medskip\medskip

   \centerline{
		\begin{tabular}{| c | c | c |}
		\hline
		  Equation  & Exponent $\ell $ of $x^\ell $ & Restrictions \\
		 \hline   
	 	 $g_0^4=h_0^2$ & $7$ &\\
	 		 $g_1^4=h_0$ & $11$ &\\
	 		  $g_0=h_2^2 $& $13$ &\\ 
	 		\hspace{.5pt}  $g_0g_{k-2}^4+g_kg_{ -2}^4 =h_0h_{k-1}^2+h_kh_{-1}^2$ & \hspace{8pt}$1+2^k$&            $0<k\neq 2,3$\\
	 	\hspace{-17pt}	  $g_{i-2}^4=0$ & $2^i+3$ & $i\ge 4$\\ 
	 	\hspace{10pt}	  $0=h_{i-1}^2$ & $2^i+5$ & $i\ge 4$\\
	 	\hline
		\end{tabular} }
	
		\medskip
 \medskip

Since $i,k\le n-1$,
  the last two equations show that only $g_0,g_1,g_{n-2},g_{n-1}$ and
 $h_0,h_1,h_2,$ $h_{n-1}$  might be nonzero.
  Moreover, 
 \begin{equation}
 \label{easy eqs}
\mbox{$ g_0^4=h_0^2, $ \  $g_1^4=h_0$ \ and \ $g_0=h_2^2.$}
  \end{equation}
 
The exponent  $1+2 ^k$, $k=n-2,$ yields $0+g_{n-2}^5=0+0$.

We need three even exponent terms in the equation $L(x)=R(x)$:
$$
\begin{array}{lllll}
g_{n-1}x^{2^{n-1}+12}=0
\\
g_1g_{1-2}^4x^{2^{1+1}}  = h_1h_{1-1}^2 x^{2^{1+1}}
 \\
(g_1g_{3-2}^4+0)  x^{2^1+2^3} =(h_1h_{3-1}^2+0) x^{2^1+2^3}.
\end{array}
$$
Then  $g_{-1}=g_{n-1}=0$, so that  $h_1 h_0=0$ by the second equation.

If $h_0=0$ then  $g_0=g_1=0$ by (\ref{easy eqs}).
If $h_1=0$ then $g_1=0$ by  the third equation, and then
$h_0=g_0=0$ by  (\ref{easy eqs}). 

Thus, $g(x)=0$ and $f(x)^3 = x^3$.  Since $n$ is odd, we obtain  $f(x)=x$, as desired.
 \end{proof}


 \noindent{\em Proof of {\rm{Theorem~\ref{NotIsomorphic}}:}}
Assume that a DHO from Example~\ref{ShadowSemifield} 
 is isomorphic to a DHO from  Example~\ref{ShadowNearFlag}.
 Let $C(a)$  be as in 
  Example~\ref{ShadowNearFlag} with spread-set
 $\Sigma  $   and shadow 
$\{ B(a)=C(a)+E_{a,a}
 \,  | \,   a\in U\}$.  By  Theorem~\ref{IsoShadow}(b),
  there is  a permutation $a\mapsto a'$ of $U$     such that 
  $0'=0$ and $\widehat{C} (a) = B(a')+E_{a,a}$
 is the canonical labeling of an additively closed spread-set.
   
 Then 
 $$
 \widehat{C} (a) =C(a') +E_{a,a} +E_{a',a'},
 $$
  where 
  $
  C(a')= (1+\sum _{i=1}^{m} c_i)a'^2{\bf 1}+ \sum _{i=1}^{m}     c_iE_{a',a'}^{(i)} 
 $
 by (\ref{DHO C}).
 Write $x\widehat{C} (a) =\sum _{i=0}^{n-1} u_i(a)x^{2^i}$ with
 each  $u_i\col F\to F$ additive (since $\widehat{C}$
   is), $u_1(a)= a'^3+ a^3$ and $u_2(a)= a'^5+a^5$ since $m \ge1$.
 The additivity of   $u_1$ and $u_2$  yield the hypotheses of 
  Lemma~\ref{Technical}.  Thus,  $a'=a$ for all $a\in U$,
so that  $\widehat{C}=C$.  
In
Example~\ref{ShadowNearFlag} we assumed that
 some
$c_j \ne0$ 
 (thereby excluding the
desarguesian spread).
By  \cite[Lemma 4.7]{DM}
 it follows that
 $\Sigma$ is not additively closed,
a contradiction.  
\qed


\section{$q$DHOs}
\label{qDHOs}

Theorem~\ref{ProjectionDHO} used orthogonal spreads over $\FFF_2$   to obtain DHOs.  This suggests the question:  what happens if larger fields are allowed.  This 
then  motivates
  the following in all characteristics:

\begin{definition}  
A set {\bf D}   of  $ n$-spaces in  a finite 
vector space
over  $\FFF_q$  is a {\rm$q$DHO \em of rank $n$}
 if the following hold:
\begin{itemize}
\item[\rm(a)]
$\dim (X_1 \cap X_2) = 1$ for all distinct $X_1,X_2 \in {\bf D},$   
\item[\rm(b)]
Each point of  a member of ${\bf D} $  lies in   precisely $q$ members of 
{\bf D}$,$ and  
\item[\rm(c)]
{\bf D}  spans the underlying vector space.
\end{itemize}

 \end{definition}

A 2DHO is just a DHO.  Note that
  $ |{\bf D}|= q^n$ (fix $Y\in {\bf D}$ and count the pairs $(P,X)$
  with $P$ a point of $X\in {\bf D}-\{Y\}$), 
  and the number of nonzero vectors  in 
$\bigcup_{X\in {\bf D}} X$
is $ |{\bf D}|  (q^n-1)/q= 
q^{n-1}(q^n-1)$.

There is a sharp division for DHOs
between even and odd characteristic:
 for any even $q$  and any $n>1$ there are known DHOs over $\FFF_q$ of rank $n$,
but  no DHO  has yet been found in odd characteristic.
We will provide several types of examples showing that
this division disappears for $q$DHOs.

\begin{example}
\label{DualPlane} 
It is easy to see that a $q$DHO of rank 2 is the dual
of the affine plane $AG(2,q)$.

\end{example}

The next example is the analogue of a standard construction of
DHOs over $\FFF_2$ (see \cite[Ex.~1.2(a)]{DF}). 

\begin{example}
\label{QuotientSpread}
For a spread ${\bf S} $   of $W=V(2n,q)$ for  $n>2$ and
 any prime power $q$,  
let $P$ be a point of  $Y\in {\bf S}$. 
Then it is straightforward to check that  
${\bf S} /P :=\{ \langle X,P\rangle /P \,  | \, X\in {\bf S} -\{ Y\}\}$
 is a $q$DHO of rank $n$ in $W/P$.  

\end{example}

\begin{example} 
[Compare Huybrechts \cite{Huy}]
\label{Huybrechts}
Let
 $V=V(n,q)$ and $W=V\oplus (V\wedge V ) $ for any prime power $q$.  Then
$$
\mbox{${\bf D}:= \{{ X(t)\, |\, t \in V}\},$ \ where  \  $X(t) := \{ (x, x\wedge t ) 
\,| \,  x\in V\}$},
$$
{\em is a} $q$DHO  {\em of rank $n$. } 
For distinct $s,t\in V$,  $(x,x\wedge s)=(x,x\wedge t)$ iff
$x\wedge (s-t)=0$. Thus $X(s)\cap X(t) =\{ (x,x\wedge s) | x\in \langle s-t\rangle \}$
is $1$-dimensional, and (a) follows. Also $\langle s-t\rangle =\langle s-t'\rangle $
implies that $t' \equiv at\pmod{\langle s\rangle}$ for some $a\in \FFF_q$,
and (b) follows.
Clearly  (c) holds.
 
\end{example}  

\begin{example}
\label{QuotientqDHO}
Let ${\bf D}$ be a $q$DHO of rank $n$ in $V=V(m,q)$.
Let $U$ be a subspace of $V$ such that $U\cap (X+Y)=0$
for all $X,Y\in {\bf D}$.  
Then 
${\bf D}/U:=\{ \langle X,U\rangle/U \, | \, X\in {\bf D}\}$
is a $q$DHO of rank $n$ in $V/U$,
using the   proof in \cite[Prop. 3.8]{Yo2}.

\end{example}

\begin{example} [\emph{Orthogonal} $q$DHOs] 
\label{OrthogonalqDHO}
In
 order to  generalize Theorem~\ref{ProjectionDHO},
let {\bf O}  be an orthogonal spread in $V^+(2n+2,q)$ and let $P$ be a point of   $Y\in {\bf O},$ so that 
$V :=P^\perp /P\simeq V^+ (2n,q)$.
Then  
$$
{\bf O}/P:=\big \{ 
\overline X:= \langle X\cap P ^\perp ,P  \rangle/P  \,| \,  X\in {\bf O}-\{Y\}\big\}
$$
is a {\rm $q$DHO} in $V,$  and $V=\overline X\oplus (Y/P) $ for each 
$X\in {\bf O}-\{Y\}.$
This is proved as in Section~\ref{Orthogonal DHOs}.

There are   orthogonal spreads {\bf O}  known  in $V^+(2n+2,q)$
for any odd $n>1$   whenever $q$ is a power of 2, and for  $n=3$ 
 and various odd $q$ \cite{Ka1.6, CKW, Moo} (obtained
 from ovoids via the triality map).
\end{example}  

\begin{remark}
Many of the known and better understood DHOs over $\FFF_2$ are bilinear \cite{DE} (roughly speaking, bilinear DHOs can be represented by
additively closed
 DHO-sets). Examples are   the
$2$DHOs in Example~\ref{QuotientSpread} if ${\bf S}$
is a semifield spread, the $2$DHOs in Example~\ref{Huybrechts},
and  the DHOs  in Example~\ref{transitive DHOs}.
It does not seem possible to give a useful definition for  bilinearity of  DHOs   using
  $\FFF_q$, $q>2$.
However,  our examples show that the notion of bilinearity
can be generalized to $q$DHOs for any $q$ in an obvious fashion
(i.\,e., by introducing the notion of 
``additively closed $q$DHO-sets'').

\end{remark}

\begin{remarks}  [{\emph{Analogues of previous results}}]
\label{Analogues of previous results}
Our main results have natural analogues for $q$DHOs.  

\begin{itemize}
\item[\rm(a)]
Proposition~\ref{Basic}(b) holds: we already know 
$|\bigcup_{X\in {\bf D}} X|$,
so that 
$S_V=\bigcup _{X\in {\bf D}} X \cup Y$
is the set of  {all} singular vectors in $V$.

\item[\rm(b)]
Proposition~\ref{Focus} holds when $q>2$:
$\Phi$ leaves $S_V-Y$ invariant, and then
 $\Phi$  also leaves $Y$ invariant
  as in Proposition~\ref{Focus} (though this time, since $q>2$
we can use 2-spaces that  contain exactly $q$ points of 
$S_V$ not in $Y$).

\item[\rm(c)]
The results in Sections~\ref{Coordinates and  symplectic spread-sets}-\ref{Sec 5} go through with at most minor changes.  For example, 
Theorem~\ref{ShadowSemifieldKW} becomes:
\emph{for even $q$ and odd  composite $n$  there are more than
$q^{n(\rho (n)-2)}/n^2$ pairwise inequivalent orthogonal
{\rm $q$DHO}s  in $V^+(2n,q)$  that arise from symplectic semifield spreads.}

\end{itemize}
\end{remarks}

\begin{remark}
Any two members of a  $q$DHO {\bf D}    meet in a point that lies in 
exactly $q$ members of {\bf D}.  Therefore, there is an associated design 
with $v=|{\bf D}   |=q^n$  ``points'', $k=q$ ``points'' per block, and 
exactly one   block containing any given pair of  ``points'';
 these are the same parameters as the design of points and lines of $AG(n,q)$.  It would be interesting to know whether  these designs are  ever 
isomorphic when $q>2$. 
\end{remark}

\section
{Concluding  remarks} 
\label{Concluding  remarks} \
(a)  Let $n$ be odd  and $1\leq r <n$  with $(n,r)=1$. Set
$F=\FFF _{2^n}$, $V=F\oplus  F$, and  as usual 
turn $V$
into a quadratic
$\FFF_2$-space  using  $Q(x,y)={\rm Tr}(xy)$.
For $a\in F$ define the operator $B(a)$ on $F$
by 
$$
xB(a)= ax^{2^r}+ (ax)^{2^{n-r}}.
$$
By \cite{Yo1},  $\{ B(a)\,| \, a\in F\}$
is a DHO-set of skew-symmetric operators defining an
orthogonal DHO ${\bf D}_{n,r}$.
Moreover, 
$|{\rm Aut}({\bf D}_{n,r})|=2^n(2^n-1)n$  \cite{Yo1,TY}.
Thus, by Example~\ref{Connection},
${\bf D}_{5,1}$  and ${\bf D}_{5,2}$ are not  
projections of   orthogonal spreads, and it  seems likely that
the same is true for all   ${\bf D}_{n,r}$, $n\geq 5$.
 
 \medskip 
 (b) There are few  papers 
 explicitly dealing with   the number of DHOs
 of a given rank \cite{BF, Ta1,TY,TY1, Yo1,Yo}. 
 For example,  \cite{Yo1,TY} obtained  
  approximately $c  d^2 $ non-isomorphic DHOs
  of rank $d$ over $\FFF_2$
    for some constant $c$. 
  However, many more may be known, but the isomorphism problems
  are open.  
  For example, the quotient construction of Example~\ref{QuotientSpread}
  associates to each spread ${\bf S}$   and each
  point $P$ of $V(2n,2)$ a DHO ${\bf S}/P$ in $V(2n,2)/P$.
  There are 
  very large numbers of non-isomorphic spreads and many points 
  $P$ to choose,  so that  the number of
  DHOs of this type probably explodes for large $n$.
  Unfortunately, as is the case for the DHOs arising from Theorem~\ref{ProjectionDHO}, 
  the isomorphism problem seems to be very difficult in general.

\medskip

(c) 
 For orthogonal spreads, in the situation of Definition~\ref{Orthogonal and symplectic spreads} isomorphisms 
  ${\bf O}/N \to {\bf O}'/N'$   between spreads    ``essentially'' lift to isomorphisms 
  ${\bf O} \to {\bf O}'$ sending $N\to N'$ \cite[Corollary~3.7]{Ka1}.
  We do not know if there is a corresponding general theorem of that sort for the DHOs in Theorem~\ref{ProjectionDHO}.
  The proof of Theorem~\ref{ShadowNearlyFlagKW}
 shows that such a lift occurs for isomorphisms among the DHOs appearing there.
  
  Theorem~\ref{ShadowSemifieldKW} is more interesting in this regard:  the proof shows that 
  isomorphisms  
  ${\bf O}/P \to {\bf O}'/P'$ among {\em those} DHOs lift to isomorphisms 
   ${\bf O} \to {\bf O}'$, but there does not seem to be any reason to expect that $P$ 
   must be sent to $P'$.
   
  It would be very interesting to have a theorem containing both Theorems~\ref{ShadowSemifieldKW} and 
  \ref{ShadowNearlyFlagKW} that involves such a lift of DHO-isomorphisms to orthogonal spread isomorphisms.

\medskip
 
 (d)  There are many more symplectic spreads known in $V$.
 Some cannot be described conveniently using spread-sets and yet have transitive automorphism groups and a precise determination of isomorphisms among the associated planes \cite{KW0}; others have trivial automorphism groups  \cite{Ka1.5}; and still others have not been examined at all.  The various associated DHOs seem even harder to study.
 
 Another family of examples arises  
 from symplectic semifields in a manner different from Section~\ref{Sec 4}:
 
 \begin{example}
 \label{transitive DHOs}
   Let $T\col  F\to \FFF_2$  and $\FFF_2\oplus F\oplus \FFF_2\oplus F$
  be as in Sections~\ref{Projections and lifts with coordinates} and
  \ref{Sec 4}, 
  with quadratic form $Q(\alpha, x, \beta, y)=\alpha\beta +T(xy)$.  Let $(F,+,*)$ be a symplectic semfield using $F$,
  such as one in Example~\ref{ShadowSemifield} given by $x*a=xC(a)$.
  Then \cite[Lemma 2.18]{KW} contains   an orthogonal spread ${\bf O}:=\{{\bf O}[s]  \,| \, s\in F  \}\cup \{{\bf O}[\infty]  \}$,
  with 
\begin{equation*}   
  \begin{split} 
\label{the orthogonal spread}
 {\bf O}[\infty] & =0\oplus0\oplus \FFF_2 \oplus  F  
\\
 {\bf O}[s] &=\big\{\big(\alpha ,x,T( xs),x*s+s (\alpha+T( xs))
 \big)  \,  \big| \,   \alpha\in \FFF_2, x\in F\big\} ,
\end{split} 
\end{equation*}  
admitting the transitive elementary abelian group consisting of the operators
$
(\alpha,x,\beta,y ) \mapsto 
(\alpha+T(x t ),   x,   \beta+T(xt), 
y+x*t+(\alpha+\beta)t),$  $t\in F$.

If $\mu\in F$ and $P_\mu:=\langle (0,0,0,\mu)\!\>\rangle $, then   
  Theorem~\ref{ProjectionDHO}  produces a  {\rm DHO}~
\emph{${\bf O}/P_\mu$  in $P_\mu^\perp/P_\mu$  
admitting a transitive elementary abelian group} induced by the above operators.

The number of DHOs obtained this way is the number of symplectic semifields of order $2^n$ multiplied by  $|F|=2^n$.  
 We conjecture that the number of pairwise inequivalent DHOs obtained is greater than the number of pairwise non-isotopic presemifields used.
  \end{example}


 (e) Orthogonal 
 DHOs (and spreads) are implicitly used 
  in \cite[Thm. 2]{CHRSS}
   to construct  
 Grassmannian packings.
  

\smallskip 

{\obeylines \small
Ulrich Dempwolff,   Department of Mathematics, Universit\"at Kaiserslautern, 67653 Kaiserslautern, Germany
e-mail: dempwolff@mathematik.uni-kl.de
\smallskip\smallskip

William M. Kantor, Department of Mathematics, University of Oregon, Eugene, OR 97403, and  College of Computer and Information Science, Northeastern University, Boston, MA 02115
email: kantor@uoregon.edu
}

\end{document}